\documentclass[12pt]{article}
\usepackage[utf8]{inputenc}
\usepackage{amsmath, amssymb, amsthm}
\usepackage{enumitem}
\usepackage{geometry}
\usepackage{setspace}
\usepackage{graphicx}
\usepackage{hyperref}
\usepackage{tikz,tikz-cd}
\usepackage{float}
\usepackage{verbatim}

\usepackage{biblatex}
\renewcommand{\b}[1]{\textbf{#1}}

\newcommand{\Z}{\mathbb{Z}}

\newcommand{\C}{\mathbb{C}}

\renewcommand{\bar}[1]{\overline{#1}}
\renewcommand{\sl}{\mathfrak{sl}}

\DeclareMathOperator{\Spec}{Spec}
\DeclareMathOperator{\Lie}{Lie}
\DeclareMathOperator{\rk}{rk}

\newtheorem{theorem}{Theorem}[section]
\newtheorem{corollary}[theorem]{Corollary}
\newtheorem{lemma}[theorem]{Lemma}

\newtheorem{proposition}[theorem]{Proposition}

\theoremstyle{definition}
\newtheorem{definition}[theorem]{Definition}

\newtheorem{example}[theorem]{Example}

\theoremstyle{remark}
\newtheorem{remark}[theorem]{Remark}

\newcommand{\g}{\mathfrak{g}}
\renewcommand{\P}{\mathcal{P}}
\newcommand{\M}{\mathcal{M}}
\newcommand{\ch}{\mathrm{ch}}
\newcommand{\N}{\mathcal{N}}
\newcommand{\h}{\mathfrak{h}}
\DeclareMathOperator{\Hom}{Hom}

\newcommand{\gl}{\mathfrak{gl}}
\newcommand{\q}{\textbf{q}}
\renewcommand{\d}{\textbf{d}}

\title{A Hesselink-type formula for the nilpotent cone of Lie algebra representations}
\author{Vasily Krylov and Frank Wang}

\begin{document}

\maketitle

\begin{abstract}
    A well-known result of Hesselink gives a formula for the $q$-character of the nilpotent cone of a semisimple Lie algebra in terms of a $q$-analog of the Kostant partition function. For a reductive Lie algebra $\g$, we define a class of \emph{Hesselink-type representations}, for which we prove an analog of Hesselink's formula under certain additional freeness assumptions. Moreover, we prove that the formula still holds for certain examples of Hesselink-type representations without the freeness assumptions. Using these methods, we obtain $q$-character formulas for the nilpotent cone of a representation of a cyclic quiver with equal dimensions, a representation of a cyclic quiver with two vertices, and a representation of a product of copies of $\sl_2$ we call an \emph{extended quiver representation of trivial type}. In the process, we give a general framework for proving formulas of this type for representations that are not necessarily Hesselink-type. We also provide some counterexamples to natural questions regarding Hesselink-type representations.
\end{abstract}

\section{Introduction}

Let $G$ be a connected complex reductive Lie group, $\g$ its Lie algebra, and let $\N$ be its nilpotent cone. Then the ring of functions $\C[\N]$ has a natural grading via the usual degree of polynomials (coming from the degree of polynomials in $\C[\g]$), and one may consider the \emph{$q$-character} of $\C[\N]$, given by

\[\ch_q\C[\N]:=\sum_{d\geq 0}\ch(\C[\N]_d)q^d,\]
where $\C[\N]_d$ is the degree $d$ component of $\C[\N]$ and $\ch(\C[\N]_d)$ is the character of $\C[\N]_d$ as a $\g$-module. A well-known formula due to Hesselink \cite{hesselinkalg} for this $q$-character computes the multiplicity of an irreducible character in terms of a $q$-analog of the dimension of the zero weight space of the corresponding irreducible representation; see also \cite{mcgovern1989}, \cite[Proposition 4.0.1]{MasonBrown2022RegularFO}.

In this paper, we prove a generalization of this formula for the nilpotent cone of a few classes of representations of Lie algebras. Namely, let $V$ be a finite dimensional representation of $G$, and let $\C[V]^G_+\subset\C[V]$ be the augmentation ideal of positive degree elements of $\C[V]^G$. Then one may define the \textbf{nilpotent cone of $V$} as $\N_V:=\Spec (\C[V]/\C[V]^G_+\C[V])$. Note that the $\C$-points of $\N_V$ are exactly the points $p$ such that the closure of the $G$-orbit of $p$ contains $0$. Then one may similarly define a grading on $\C[\N_V]$ coming from the usual grading on $\C[V]$, so we may consider $q$-characters in this setting as well. 

Following \cite{heaton}, we further consider a multigraded analog of the $q$-character. Write $V=\bigoplus_{k=1}^lV_k$ where each $V_k$ is a simple $G$-module. Then there is a $(\C^\times)^l$-action on $V$ where the $k$-th copy of $\C^\times$ acts by scaling $V_k$. This allows us to define a grading on $\C[V]$ by $l$-tuples of nonnegative integers. Namely, the $\d=(d_1,\dots,d_l)$-th graded component $\C[V]_{\d}$ of $\C[V]$ is the component on which $(z_1,\dots,z_l)\in(\C^\times)^l$ acts by the scalar $\prod_{k=1}^lz_k^{d_k}$. We define the \textbf{$\q$-character} of $\C[V]$ as
\[\ch_{\q}\C[V]=\sum_{\d\in\Z_{\geq 0}^l}\ch(\C[V]_{\d})\q^{\d}\in K_0(\mathrm{Rep}(G))[[\q]]\]
where $\q=(q_1,\dots,q_l)$ and $\q^{\d}=\prod_{k=1}^lq_k^{d_k}$. The augmentation ideal $\C[V]^G_+\subset\C[V]$ is graded with respect to the $(\C^\times)^l$-action, so the grading on $\C[V]$ descends to a grading on $\C[\N_V]$, and so we also have a notion of $\q$-character for $\C[\N_V]$. Note that if we restrict the $(\C^\times)^l$-action to the diagonal copy of $\C^\times$, then we recover the ordinary $q$-character of $\C[\N_V]$.

Let us fix some notation. Fix a Cartan subalgebra $\h\subset\g$ and let $T\subset G$ be the torus with $\Lie T=\h$. Let $W$ be the Weyl group of $\g$ and $R$ be the root system of $\g$. Fix a set of positive roots $R_+$ and let $\rho=\frac{1}{2}\sum_{\alpha\in R_+}\alpha$, $\rho^\vee=\frac{1}{2}\sum_{\alpha\in R_+}\alpha^\vee$. Let $\Lambda$ be the weight lattice of $\g$ and let $\Lambda_+$ be the set of dominant weights. For $\lambda\in\Lambda_+$, let $L_\lambda$ be the irreducible representation of $\g$ with highest weight $\lambda$ and let $\chi_\lambda$ be its character. Let $\langle-,-\rangle$ be the natural pairing $\h^*\otimes\h\to\C$.

When $\g=\sl_2$, let us identify $\Lambda$ with $\Z$, so that the fundamental weight of $\g$ is $1$.

Let $S(V)$ be the multiset of the weights in $V$, and let 
\begin{align*}
    S_+(V)&:=\{\alpha\in S\mid\langle\alpha,\rho^\vee\rangle>0\},\\
    S_0(V)&:=\{\alpha\in S\mid\langle\alpha,\rho^\vee\rangle=0\},\\
    S_-(V)&:=\{\alpha\in S\mid\langle\alpha,\rho^\vee\rangle<0\}
\end{align*}
where each of $S_+(V),S_0(V),S_-(V)$ is viewed as a submultiset of $S(V)$. Also, let $S_{\leq 0}(V):=S_-(V)\cup S_0(V)$. When $V$ is clear from context, we will omit it from the notation. When we say a function on one of $S,S_+,S_0$, or $S_{\leq 0}$, we mean a function on all occurrences of weights in the multiset, not just on distinct weights.

Let $(S_0)_0$ be the set of all functions $m\colon S_0\to\Z_{\geq 0}$ such that $m\neq 0$ and $\sum_{\alpha\in S_0}m(\alpha)\alpha=0$. Let us call a function $m:S\to\Z_{\geq 0}$ \textbf{weakly negative} if $m(\alpha)=0$ for all $\alpha\in S_+$ and for all $m'\in (S_0)_0$, there is some $\alpha\in S_0$ such that $m(\alpha)<m'(\alpha)$. For a function defined on a sub(multi)set of $S$, we say it is \textbf{weakly negative} if its extension to $S$ by zero is weakly negative.

Let $S_k(V)$ be the subset of $S(V)$ corresponding to weights of $V_k$. Consider the following $\q$-analog of the Kostant partition function:
\[\P_{\q,V}\colon\Lambda\to\Z[\q],\]
\[
    \P_{\q,V}(\lambda)=\sum_{d\in\Z_{\geq 0}^l}\#\Big\{m\colon S\to\Z_{\geq 0}\mid\sum_{\alpha\in S}m(\alpha)\alpha=-\lambda,\sum_{\alpha\in S_k}m(\alpha)=d_k,\text{ $m$ is weakly negative}\Big\}q^\d.
\]

A priori, this definition of $\P_{\q,V}$ only lands in $\Z[[\q]]$, but it in fact lands in $\Z[\q]$ due to Dickson's lemma. When $V$ is clear from context, we will let $\P_\q:=\P_{\q,V}$.

Note that when $V$ is the adjoint representation, $S_-$ is the set of negative roots $R_-$ and $S_0$ is a multiset of $\rk\,\g$ zeroes, corresponding to $\h\subset\g$. In particular, $(S_0)_0$ is all nonzero functions $S_0\to\Z_{\geq 0}$, and a weakly negative function is simply a function $R_-\to\Z_{\geq 0}$. In this case $\P_\q$ reduces to the $q$-analog of the Kostant partition function introduced by Lusztig \cite{lusztigqkostant}.

\begin{example}\label{eg:1,1}
    Consider the case of a cyclic quiver with two vertices that are both one-dimensional. More explicitly, let $K_1,K_2$ be complex vector spaces with $\dim K_1=\dim K_2=1$. Then let $\g\subset\sl(K_1\oplus K_2)$ be the Lie algebra that fixes $K_1,K_2$, and let $V_1=\Hom(K_1,K_2), V_2=\Hom(K_2,K_1),V=V_1\oplus V_2$. Then we have $\g\cong\gl_1$, so we have $\rho^\vee=0$. The weight lattice $\Lambda$ is $\Z\subset\h^*\cong\C$, and $V$ has two weights which we can write as $\pm 1\in\Lambda$. Since $\rho^\vee=0$ we have $S=S_0=\{\pm 1\}$. So a weakly negative function is a function $m\colon\{\pm 1\}\to\Z_{\geq 0}$ such that either $m(1)=0$ or $m(-1)=0$. We thus obtain $\P_\q(k)=\begin{cases}
        q_1^k&k\geq 0\\
        q_2^{-k}&k\leq 0
    \end{cases}$ for any integer $k$. 
\end{example}

Also, define
\[\M_\q\colon \Lambda\to\Z[\q],\quad \M_\q(\lambda)=\sum_{w\in W}(-1)^{\ell(w)}\P_\q(w(\lambda+\rho)-\rho).\]

This is a $\q$-analog of the Kostant multiplicity formula for the zero weight space of a representation that also generalizes the corresponding formula in \cite{lusztigqkostant}.

We will prove the following formula for certain classes of representations:

\begin{equation}\label{eq:formula}
    \ch_\q\C[\N_V]=\sum_{\lambda\in\Lambda_+}\M_\q(\lambda)\chi_\lambda.
\end{equation}

Note that when $V$ is the adjoint representation of $\g$, $\N_V$ is the usual nilpotent cone of $\g$, and Formula \eqref{eq:formula} recovers the classical formula of Hesselink described at the beginning for the $q$-character of functions on the nilpotent cone. 

\subsection{Hesselink-type representations}

For a finite dimensional $G$-representation $V$, let $V^{\rho^\vee}$ be the invariants in $V$ under the action of $\rho^\vee\in\h$. In other words, $V^{\rho^\vee}$ is the direct sum of all weight spaces in $V$ whose weights belong to $S_0$. Monomials in $\C[V^{\rho^\vee}]$ correspond to functions $S_0\to\Z_{\geq 0}$, monomials in $\C[V^{\rho^\vee}]^\h$ correspond to functions in $(S_0)_0$, and monomials surviving in $\C[V^{\rho^\vee}]/\C[V^{\rho^\vee}]^\h_+\C[V^{\rho^\vee}]$ are exactly the ones corresponding to weakly negative functions on $S_0$. Then we introduce the following new class of $G$-representations, motivated by the proof of Theorem \ref{thm} below.

For a weight $\alpha\in S_k(V)$, let $q_\alpha:=q_k$. Also, for a $\Z_{\geq 0}^l$-graded vector space $Z$ we let $\dim_\q Z:=\sum_{\d\in\Z_{\geq 0}^l}\dim(Z_\d)\q^\d$ be the graded dimension (or Hilbert series) of $Z$. 

\begin{definition}\label{def:hesselink}
    A \textbf{Hesselink-type} representation is a representation $V$ such that
    \[X_V:=\frac{\sum_{w\in W}(-1)^{\ell(w)}t^{w\rho}\prod_{\alpha\in S_+}(1-q_\alpha t^{-w\alpha})}{\sum_{w\in W}(-1)^{\ell(w)}t^{w\rho}}\]
    is a $\Z[\q]$-multiple of the trivial character.
\end{definition}

All representations for which we prove \eqref{eq:formula} are Hesselink-type.

A finite dimensional $G$-representation $V$ is called \textbf{cofree} if $\C[V]$ is free over $\C[V]^G$ and \textbf{coregular} if $\C[V]^G$ is a free commutative algebra. It follows from standard commutative algebra that all cofree representations are coregular (see \cite{book}, Chapter 8.1). Much of the significance of Hesselink-type representations is via the following proposition:

\begin{proposition}\label{prop:hesselink2}
    If $V$ is Hesselink-type and cofree, and $\C[V^{\rho^\vee}]$ is free over $\C[V^{\rho^\vee}]^\h$, then $V$ satisfies Formula \eqref{eq:formula}.
\end{proposition}

Note that it is not true that every Hesselink-type representation is cofree (see Example \ref{eg:noncofree}) or even coregular (see Example \ref{eg:noncoregular}), or that every Hesselink-type representation has $\C[V^{\rho^\vee}]$ free over $\C[V^{\rho^\vee}]^\h$ (see e.g.~Proposition \ref{prop:cycle}). It is also not true that every Hesselink-type representation satisfies Formula \eqref{eq:formula}. Nonetheless, for most Hesselink-type representations that we found, they appear to satisfy the formula in Theorem \ref{thm} even if one or more of the assumptions in Proposition \ref{prop:hesselink2} does not hold. It would be an interesting question to characterize what distinguishes Hesselink-type representations for which the $\q$-character formula holds and ones where it does not.

Our main theorem is the following:

\begin{theorem}\label{thm}
    Let $G,V$ be one of the following:
    \begin{enumerate}
        \item $V$ is the adjoint representation of $\g$.\label{eg:adj}
        \item Let $K_1,\dots,K_l$ be a collection of complex vector spaces of the same dimension $N$, and let $G=\prod_{k\in\Z/l\Z}GL(K_k),V=\bigoplus_{k\in\Z/l\Z}\Hom(K_k,K_{k+1})$. This is the representation of a cyclic quiver of length $l$ where the vector spaces assigned to each vertex have equal dimensions.\label{eg:eqdim}
        \item Let $K_1,K_2$ be complex vector spaces (of possibly different dimensions), and let $G=GL(K_1)\times GL(K_2),V=\Hom(K_1,K_2)\oplus\Hom(K_2,K_1)$. This is the representation of a cyclic quiver of length 2.
        \item Let $\g$ be a product of finitely many copies of $\sl_2(\C)$, and $V$ be one of the representations described in Section \ref{sec:sl2}.
    \end{enumerate}
    Then $V$ is Hesselink-type and Formula \eqref{eq:formula} holds for $V$.
\end{theorem}

The proof of this theorem for most of the cases relies on Proposition \ref{prop:hesselink2}. Note that all representations covered in Theorem \ref{thm} are cofree; the adjoint representation is cofree by Kostant's theorem, and all representations of cyclic quivers are cofree as they fall under the class of $\Z/l\Z$-graded Lie algebras, and all such representations are cofree (\cite{vinberg}, \cite{book}). We will further elaborate on this point when we discuss these cases directly in Sections \ref{sec:eqdim} and \ref{sec:2vertex}. For the case $\g=\prod\sl_2$, we prove cofreeness directly in Section \ref{sec:sl2}.

\begin{example}\label{eg:1,1N}
    Let us revisit the case of Example \ref{eg:1,1}. If $v_+,v_-$ are basis vectors of $V_1^*$, $V_2^*$, respectively, then $\C[V]^G=\C[v_+v_-]$, and we have $\C[\N_V]=\C[v_+,v_-]/(v_+v_-)$. Choose an isomorphism $\h^*\cong\C$ so that $v_+$ is the vector of weight $1$ and $v_-$ is the vector of weight $-1$. Since $R$ is empty, all weights are dominant. We have $W=1$, so $\M_\q(k)=\P_\q(k)$ and Formula \eqref{eq:formula} reads
    \[\ch_\q\C[\N_V]=\sum_{k\geq 0}q_1^k\chi_k+\sum_{k<0}q_2^{-k}\chi_k.\]
    Indeed, for nonnegative $k$ the unique copy of $\chi_k$ in $\C[\N_V]$ is spanned by $v_+^k$ in degree $k$, and the unique copy of $\chi_{-k}$ in $\C[\N_V]$ is spanned by $v_-^k$ in degree $k$.

    Technically, Theorem \ref{thm} only considers the case where we take $\g$ to be the subspace of $\gl(K_1\oplus K_2)$ that fixes $K_1,K_2$ rather than $\sl(K_1\oplus K_2)$, but the theorem carries over to the latter since the copy of $\C$ in $\gl(K_1\oplus K_2)\cong\sl(K_1\oplus K_2)\oplus\C$ acts by zero on $V$. This also holds for all cyclic quivers covered in Theorem \ref{thm}.
\end{example}

\begin{remark}
    Outside the classes of representations encompassed by Theorem \ref{thm}, we also found similar-looking formulas for $\ch_\q\C[\N_V]$ for some other representations $V$, which are described in Section \ref{sec:nonexamples}. It would be interesting to see if there is a ``universal" formula that generalizes all formulas of this type.
\end{remark}

\subsection{Previous work}

There have been a number of $q$-character formulas obtained for various classes of representations following Hesselink's work. 

In the setting of Kostant-Rallis \cite{Kostant1971OrbitsAR}, combinatorial formulas for $\ch_q\C[\N_V]$ were obtained in \cite{howetanwillenbring}, \cite{frohmader} in terms of semistandard Young tableaux. Similar methods were also used in \cite{heaton} to obtain a combinatorial formula for $\ch_\q\C[\N_V]$ for the representation of a cyclic quiver with all dimensions equal to 2 and in \cite{frohmaderheaton} to obtain a formula for $\ch_q\C[\N_V]$ up to bounded degree in $q$. Although both our formula and the formulas in \cite{heaton}, \cite{frohmaderheaton} handle representations of cyclic quivers, our formula is fundamentally different from these formulas, and comparison between them yields combinatorial relationships which we briefly describe in Section \ref{sec:heaton} and Remark \ref{rem:frohmaderheaton}.

In \cite{fox}, Fox computes a $\q$-character formula for all cofree representations\footnote{The paper actually only considers ordinary $q$-graded characters for $\Z/l\Z$-graded Lie algebras, but the formula immediately generalizes to multigraded characters for cofree representations with the same arguments.}, a much broader class than the representations considered in Theorem \ref{thm}. However, the formula obtained there is a more complicated formula that requires computation of the character of $\C[V]^G$ (or the exponents of $V$ in the case where $V$ is coregular) and involves more terms than our formula. In the process of proving our formula, we actually rederive a form of Fox's formula in \eqref{eq:fox}.

Finally, in \cite{wallachwillenbring}, an approach similar to ours was used to compute a $q$-character formula in the Kostant-Rallis setting. To our knowledge, this is the only other $q$-character formula in the literature that is a direct generalization of Hesselink's original formula, i.e.~that reduces to Hesselink's formula for the adjoint representation. In fact, their formula falls under a more general class of $\q$-character formulas analogous to Formula \eqref{eq:formula} described in Section \ref{sec:nonexamples}.

\subsection{Motivation}

This paper was originally motivated by $D$-modules on the cyclic quiver and their relation to representations of the rational Cherednik algebra. Namely, let $V$ be the representation of a cyclic quiver with equal dimensions as in Case \ref{eg:eqdim} of Theorem \ref{thm} and let $D^G(V)$ be the category of $G$-equivariant $D$-modules on $V$. Consider the group $G(l,1,N):=(\Z/l\Z)^N\rtimes S_N$ and let $t,c,d_1,\dots,d_l$ be complex numbers. Then one has the \emph{rational Cherednik algebra} $H_{t,c,d}(G(l,1,N))$ (see \cite{rationalcherednik}), and in \cite{montarani} Montarani produced a functor $F_{N,\chi}:D^G(V)\to\text{Rep}\,H_{t,c,d}(G(l,1,N))$. Through this functor one may study graded characters of $H_{t,c,d}(G(l,1,N))$-modules by studying the $q$-characters of equivariant $D$-modules on $V$. 

The simplest case of an equivariant $D$-module on $V$ is the space of global functions $\C[V]$. In this case the $q$-character of $\C[V]$ as an equivariant $D$-module is exactly the $q$-character of $\C[V]$ as described above, and a formula for $\ch_q\C[V]$ is equivalent to a formula for $\ch_q\C[\N_V]$ via Lemma \ref{lem:cofreeiso}. We first derived Formula \eqref{eq:formula} for this case, and after discussions with Pavel Etingof and Zhiwei Yun we realized it would be natural to consider generalizations of this formula for other groups and representations. 

In particular, much of this paper is a result of an attempt to generalize our formula to the $q$-character of the nilpotent cone of a $\Z/l\Z$ graded Lie algebra, following the work of Lusztig-Yun and others on graded Lie algebras (\cite{wille},\cite{LUSZTIG1995147},\cite{lusztigyun}). More explicitly, if $\tilde{\g}=\bigoplus_{i\in\Z/l\Z}\g_i$ is a graded Lie algebra with $[\g_i,\g_j]\subset\g_{i+j}$ and $\tilde{\N}$ is the nilpotent cone of $\tilde{\g}$, then the nilpotent cone of $\g_1$ is defined as $\g_1\cap\tilde{\N}$. Setting $\g=\g_0$ and $V=\g_1$, the nilpotent cone of $\g_1$ is exactly the nilpotent cone $\N_{\g_1}$ that we are considering in this paper. However, we found a few graded Lie algebras for which Formula \eqref{eq:formula} does not hold, which we have included in Section \ref{sec:counter}. In particular, not all representations arising from graded Lie algebras are Hesselink-type.

\subsection{Structure of the paper}

In Section \ref{sec:id}, we reduce the formula in Theorem \ref{thm} to an identity for cofree representations and use this to prove Proposition \ref{prop:hesselink2}. In Section \ref{sec:cases}, we use the machinery of Section \ref{sec:id} to prove Theorem \ref{thm} in the first three cases. In Section \ref{sec:sl2}, we define another class of representations called \emph{extended quiver representations of trivial type}, and prove Theorem \ref{thm} for them. We also show that they exhaust all Hesselink-type representations when $\g$ is a product of copies of $\sl_2$. In Section \ref{sec:small}, we give a geometric interpretation of Hesselink-type representations, give some examples of formulas similar to Formula \eqref{eq:formula} for non-Hesselink-type representations, and provide a few counterexamples to potential connections Hesselink-type representations could have had to other classes of representations.

\subsection*{Acknowledgements}

We are very grateful to Pavel Etingof, who connected the two authors and with whom we had many useful discussions. We are also especially grateful to Zhiwei Yun, whose discussions with the second author led us to consider both $\Z/l\Z$-graded Lie algebras and the case where $G$ is a product of $SL_2$'s as in Section \ref{sec:sl2}. We are also grateful to Kenta Suzuki and Luke Robitaille for useful discussions. Some of the results in the paper were obtained in collaboration with ChatGPT 5.6 Pro. Specifically, the results obtained with essential AI contribution are Lemma \ref{lem:hesselinkid}, Lemma \ref{lem:leafdisjoint}, Proposition \ref{prop:cycle}, Section \ref{sec:sl22}, and Section \ref{sec:counter}. All writing was done by humans.

The first author was supported by the
Simons Foundation Award 888988 as part of the Simons Collaboration on Global Categorical
Symmetries. The second author was supported by the National Science Foundation Graduate Research Fellowship under Grant No.~2141064. Any opinions, findings, and conclusions or recommendations expressed in this material are those of the author(s) and do not necessarily reflect the views of the National Science Foundation.

\section{Reduction of formula}\label{sec:id}

In this section, we reduce Formula \eqref{eq:formula} to an identity in terms of the weights of $V$ that will be central to the rest of the paper and serves as the basis for the definition of Hesselink-type representations. For Hesselink-type representations, this identity looks particularly nice, and allows us to prove Proposition \ref{prop:hesselink2}.

The following is inspired by \cite{hesselinkalg}, where it is proved for the adjoint representation.

\begin{proposition}\label{prop:hesselink}
    Let $V$ be a cofree representation of $G$, and let $P\in\C(\q)[t^{-\alpha}]_{\alpha\in S}$. Suppose there is some $M_\q:\Lambda_+\to\Z[\q]$ such that
    \[\sum_{w\in W}(-1)^{\ell(w)}t^{w\rho}\sum_{\lambda\in\Lambda_+}t^{w\lambda}M_\q(\lambda)=\sum_{w\in W}(-1)^{\ell(w)}t^{w\rho}w\left(\frac{P}{\prod_{\alpha\in S}(1-q_\alpha t^{-\alpha})}\right).\]
    Then we have
    \[\ch_\q\C[\N_V]=\sum_{\lambda\in\Lambda_+}M_\q(\lambda)\chi_\lambda\]
    if and only if
    \begin{equation}
        \frac{\sum_{w\in W}(-1)^{\ell(w)}t^{w\rho}}{\dim_\q\C[V]^G}=\sum_{w\in W}(-1)^{\ell(w)}t^{w\rho}wP.\label{eq:hesselink}
    \end{equation}
\end{proposition}

Before we can prove the proposition, we need the following auxiliary lemma. 

\begin{lemma}\label{lem:cofreeiso}
    Let $V$ be cofree. Then there exists an isomorphism of graded $G$-modules $\C[\N_V]\otimes\C[V]^G\cong\C[V]$.
\end{lemma}

\begin{proof}
    Since $V$ is cofree, we may write $\C[V]=\bigoplus_i\C[V]^GL_{\lambda_i}$ for some (infinitely many, not necessarily distinct) irreducible $G$-representations $L_{\lambda_i}$. Let $f\in\C[V]_+^G\C[V]$, and write $f=\sum_{r=1}^Ns_rg_r$ for $s_r\in\C[V]^G_+$. Writing each $g_r$ in terms of the $L_{\lambda_i}$, we see that $f\in\bigoplus_i\C[V]^G_+L_{\lambda_i}$, and conversely we have $\bigoplus_i\C[V]^G_+L_{\lambda_i}\subset\C[V]^G_+\C[V]$, so we have $\bigoplus_i\C[V]^G_+L_{\lambda_i}=\C[V]^G_+\C[V]$. It follows that as a graded $G$-module we have $\C[\N_V]\cong\bigoplus_iL_{\lambda_i}$, and we have an isomorphism $\bigoplus_iL_{\lambda_i}\otimes\C[V]^G\to\C[V]$ given by the multiplication map, so the result follows.
\end{proof}

\begin{proof}[Proof of Proposition \ref{prop:hesselink}]
    Let $v_1,\dots,v_n$ be a basis of $V$ consisting of weight vectors, such that each $v_i$ has weight $\alpha_i$. Then as a $T$-representation, we have $\ch_\q\C[v_i^*]=\frac{1}{1-q_\alpha t^{-\alpha_i}}$. Thus we get
    \[\ch_\q\C[V]=\ch_\q\left(\bigotimes_{i=1}^n\C[v_i^*]\right)=\prod_{i=1}^n\ch_\q\C[v_i^*]=\frac{1}{\prod_{\alpha\in S}(1-q_\alpha t^{-\alpha})}.\]
    Let $\ch_\q\C[\N_V]=\sum_{\lambda\in\Lambda_+}c_\q(\lambda)\chi_\lambda$ for $c_\q:\Lambda_+\to\Z_{\geq 0}[\q]$. By the Weyl character formula and Lemma \ref{lem:cofreeiso}, we have
    \begin{equation}\label{eq:fox}
        \sum_{w\in W}(-1)^{\ell(w)}\sum_{\lambda\in\Lambda_+}t^{w(\lambda+\rho)}c_\q(\lambda)=\prod_{\alpha\in S}\left(\frac{1}{1-q_\alpha t^{-\alpha}}\right)\frac{\sum_{w\in W}(-1)^{\ell(w)}t^{w\rho}}{\ch_\q\C[V]^G}.
    \end{equation}
    Now, using \eqref{eq:hesselink} this is equal to
    \[\prod_{\alpha\in S}\left(\frac{1}{1-q_\alpha t^{-\alpha}}\right)\sum_{w\in W}(-1)^{\ell(w)}t^{w\rho}wP=\sum_{w\in W}(-1)^{\ell(w)}t^{w\rho}wP\prod_{\alpha\in S}\left(\frac{1}{1-q_\alpha t^{-w\alpha}}\right)\]
    since $\prod_{\alpha\in S}\frac{1}{1-q_\alpha t^{-\alpha}}$ is $W$-invariant. Now, the result follows from the definition of $M_\q$. Reversing all equalities proves the converse.
\end{proof}

If we take $P=\frac{\prod_{\alpha\in S_+}(1-q_\alpha t^{-\alpha})}{\dim_\q\C[V^{\rho^\vee}]^\h}$ in Proposition \ref{prop:hesselink}, then \eqref{eq:hesselink} becomes
\begin{equation}
    X_V=\frac{\dim_\q\C[V^{\rho^\vee}]^{\h}}{\dim_\q\C[V]^G}\label{eq:id}
\end{equation}
where $X_V$ was defined in Definition \ref{def:hesselink}. In this case, we have the following:

\begin{lemma}\label{lem:hesselinkid}
    Let $V$ be a Hesselink-type representation. Then $V$ satisfies \eqref{eq:id}.
\end{lemma}

\begin{proof}
    Since $V$ is Hesselink-type, we may write $X_V=x_V$ for some $x_V\in\Z[\q]$. Then rearranging terms, we have
    \[\sum_{w\in W}(-1)^{\ell(w)}t^{w\rho}\prod_{\alpha\in S_+}(1-q_\alpha t^{-w\alpha})=x_V\sum_{w\in W}(-1)^{\ell(w)}t^{w\rho}.\]
    Multiplying both sides by $\ch_\q\C[V]=\frac{1}{\prod_{\alpha\in S}(1-q_\alpha t^{-\alpha})}$ (see the proof of Proposition \ref{prop:hesselink}), we obtain
    \[\sum_{w\in W}(-1)^{\ell(w)}\frac{t^{w\rho}}{\prod_{\alpha\in S_{\leq 0}}(1-q_\alpha t^{-w\alpha})}=x_V\ch_\q\C[V]\sum_{w\in W}(-1)^{\ell(w)}t^{w\rho}\]
    since $\ch_\q\C[V]$ is $W$-invariant.
    
    Let us now compute the coefficient of $t^\rho$ on both sides. By the Weyl character formula, the coefficient of $t^\rho$ in $\ch_\q\C[V]\sum_{w\in W}(-1)^{\ell(w)}t^{w\rho}$ is given by $\dim_\q\C[V]^G$. On the left hand side, the summand corresponding to $w\in W$ contains terms whose exponent in $t$ is given by $w\rho-\sum_{\alpha\in S_{\leq 0}}m(\alpha) w\alpha$ for some $m:S_{\leq 0}\to\Z_{\geq 0}$. Thus any contribution to the coefficient of $t^\rho$ must come from an $m$ such that $\sum_{\alpha\in S_{\leq 0}}m(\alpha)\alpha=\rho-w^{-1}\rho$. But if $w\neq 1$, then $\langle \rho-w^{-1}\rho, \rho^\vee\rangle>0$, while $\langle\sum_{\alpha\in S_{\leq 0}}m(\alpha)\alpha, \rho^\vee\rangle$ is always nonpositive since $\langle\alpha,\rho^\vee\rangle\leq 0$ for all $\alpha\in S_{\leq 0}$ by definition. So the only summand that contributes to the coefficient of $t^\rho$ is $w=1$, and for this summand we clearly have that the contribution to the coefficient of $t^\rho$ is exactly given by $\dim_\q\C[V^{\rho^\vee}]^\h$. So comparing coefficients of $t^\rho$, we must have $x_V=\frac{\dim_\q\C[V^{\rho^\vee}]^\h}{\dim_\q\C[V]^G}$.
\end{proof}

Note that this lemma does not require any cofreeness assumption. It also allows us to prove Proposition \ref{prop:hesselink2}.

\begin{proof}[Proof of Proposition \ref{prop:hesselink2}]
    By Lemma \ref{lem:hesselinkid}, $V$ satisfies \eqref{eq:id}. Take $P=\frac{\prod_{\alpha\in S_+}(1-q_\alpha t^{-\alpha})}{\dim_\q\C[V^{\rho^\vee}]^\h}$. When $\C[V^{\rho^\vee}]$ is free over $\C[V^{\rho^\vee}]^\h$, we obtain by the definition of $\P_\q$ that
    \[\sum_{\lambda\in\Lambda}\P_\q(\lambda)t^\lambda=\frac{P}{\prod_{\alpha\in S}(1-q_\alpha t^{-\alpha})},\]
    so in Proposition \ref{prop:hesselink} we obtain $M_\q=\M_\q$ and the result follows.
\end{proof}

\begin{remark}
    For a coregular representation $V$ one defines the \emph{exponents} $m_1,\dots,m_r$ of $V$ so that the free generators of $\C[V]^G$ have degrees $m_1+1,\dots,m_r+1$. Any cofree representation must be coregular (see Chapter 8.1 of \cite{book}), so we may rewrite the non-multigraded version of \eqref{eq:id} as
    \[X_V=\dim_q\C[V^{\rho^\vee}]^\h\prod_{i=1}^r(1-q^{m_i+1}).\]
\end{remark}

\section{Proof of Theorem \ref{thm} for the adjoint representation and representations of cyclic quivers}\label{sec:cases}

We will now prove Theorem \ref{thm} on a case-by-case basis. We relegate the case where $\g$ is a product of copies of $\sl_2$ to Section \ref{sec:sl2} since it is much more involved than the other cases.

\subsection{Proof for the adjoint representation}

This case has already been proven by Hesselink \cite{hesselinkalg}, but we give a proof in terms of the language of Section \ref{sec:id} for completeness. The adjoint representation is simple, so we may let $\q=q$. When $V$ is the adjoint representation, $S$ is a multiset consisting of $R$ and $\rk\,\g$ copies of 0 corresponding to $\h\subset V$. We have $S_+=R_+$ is the set of positive roots and $V^{\rho^\vee}=\h$. So $\dim_q\C[V^{\rho^\vee}]^\h=\dim_q\C[\h]=\frac{1}{(1-q)^{\rk\g}}$. $\C[V^{\rho^\vee}]=\C[\h]$ is free over $\C[V^{\rho^\vee}]^\h$ of rank 1. Moreover, $\dim_q\C[V]^G=\frac{1}{\prod_{i=1}^{\rk\g}(1-q^{m_i+1})}$ where $m_1,\dots,m_{\rk\g}$ are the exponents of $\g$. So \eqref{eq:id} becomes
\[X_V:=\frac{\sum_{w\in W}(-1)^{\ell(w)}t^{w\rho}\prod_{\alpha\in S_+}(1-qt^{-w\alpha})}{\sum_{w\in W}(-1)^{\ell(w)}t^{w\rho}}=\prod_{i=1}^{\rk\g}\frac{1-q^{m_i+1}}{1-q}=\sum_{w\in W}q^{\ell(w)},\]
where the last equality is a well-known formula \cite[Theorem 3.15]{humphreys1992reflection}. This identity is now exactly Equation 2.16 of \cite{macdonald}.

\subsection{Proof for the cyclic quiver with equal dimensions}\label{sec:eqdim}

Let $K_1,\dots,K_l$ be vector spaces of dimension $N$. We now consider when $G=\prod_{k\in\Z/l\Z}GL(K_k),V_k=\Hom(K_k,K_{k+1}), V=\bigoplus_{k\in\Z/l\Z}V_k$. Then $\g$ has a basis given by $\{\alpha_{i,j}^k\}_{1\leq i,j\leq N,k\in\Z/l\Z}$ where $\alpha_{i,j}^k$ has a 1 in the $(i,j)$-th entry of $\gl(K_k)$ and zeroes elsewhere, and $V_k$ has a basis given by $\{v_{i,j}^k\}_{1\leq i,j\leq N}$ where $v_{i,j}^k\in\Hom(K_k,K_{k+1})$ sends the $j$-th basis vector in $K_k$ to the $i$-th basis vector in $K_{k+1}$. A basis of $\h$ is given by $\{\alpha_{i,i}^k\}$, and a basis for the weight lattice is then given by $\{\b{e}_i^k\}$, which we define to be the dual basis to our basis of $\h$. Then $v_{i,j}^k$ are weight vectors with weight $\b{e}_i^{k+1}-\b{e}_j^k$. Identifying $\h\cong\h^*$ via the dot product with respect to the given basis of $\h$, we have $\rho=\rho^\vee=(\rho_1,\dots,\rho_l)$ where $\rho_k=\sum_{i=1}^N\frac{N+1-2i}{2}\b{e}_i^k$ is the usual element $\rho$ for $\gl_N$. So we get
\[S_+=\{\b{e}_i^{k+1}-\b{e}_j^k\mid i<j\},\quad S_0=\{\b{e}_i^{k+1}-\b{e}_i^k\},\quad S_-=\{\b{e}_i^{k+1}-\b{e}_j^k\mid i>j\}.\]
Then it is clear from this description that $\C[V^{\rho^\vee}]^\h$ is the free algebra generated by the generators $\prod_{k=1}^l(v_{i,i}^k)^*$ for $1\leq i\leq N$, as the corresponding weight is $\sum_{k\in\Z/l\Z}(\b{e}_i^k-\b{e}_i^{k+1})=0$. Thus we get $\dim_q\C[V^{\rho^\vee}]^\h=\frac{1}{(1-\prod_{k=1}^lq_k)^N}$. Moreover, it is clear that $\C[V^{\rho^\vee}]$ is free over $\C[V^{\rho^\vee}]^\h$ with generators given by monomials in the $(v_{i,i}^k)^*$ where for each $i$, there is at least one $k$ for which the degree of the monomial in $(v_{i,i}^k)^*$ is zero.

On the other hand, this case falls under the framework of $\Z/l\Z$-graded Lie algebras, as $\g,V$ are exactly the pieces $\g_0,\g_1$ of the Lie algebra $\tilde{\g}:=\gl(\bigoplus_{k\in\Z/l\Z}K_k)$ where the grading is given by $\g_i:=\bigoplus_{k\in\Z/l\Z}\Hom(K_k,K_{k+i})$. In this setting, Vinberg showed that $V$ is cofree \cite{vinberg}, so by Proposition \ref{prop:hesselink2}, to prove Theorem \ref{thm} in this case it suffices to show that $V$ is Hesselink-type.

The last piece of information we need before proving the theorem for this case is the Weyl group, which can be expressed as $W=\prod_{k\in\Z/l\Z}W_k$, where $W_k\cong S_N$ is the Weyl group of $GL(K_k)$, acting on $\h$ the usual way.

\begin{proposition}\label{prop:eqdim}
    Let $V$ be the representation associated to a cyclic quiver with equal dimensions. Then $V$ is Hesselink-type.
\end{proposition}

\begin{proof}
    It suffices to show that if $t^\lambda$ appears in the expansion of
    \[t^\rho\prod_{\alpha\in S_+}(1-q_\alpha t^{-\alpha}),\]
    then either $\lambda\in W\rho$ or $\lambda$ lies on the reflection hyperplane of some root. The expansion of this expression is given by
    \[\sum_{I\subset S_+}(-1)^{|I|}(\q)^{|I|}t^{\rho-\sum_{\alpha\in I}\alpha},\]
    where by $(\q)^{|I|}$ we mean that the exponent of $q_k$ is $|I\cap S_k|$. So the $\lambda$ that appear are exactly $\lambda_I:=\rho-\sum_{\alpha\in I}\alpha$ for $I\subset S_+$. We will show that for all such $I$, either $\lambda_I\in W\rho$ or $\lambda_I$ lies on the reflection hyperplane of some root. Fix $k\in\Z/l\Z$, and choose any $1\leq i\leq N$. Consider the component of all relevant weights in a given $\b{e}_i^k$ in the basis of the weight lattice. We have $\langle\rho,\b{e}_i^k\rangle=\frac{N+1-2i}{2}$, and the elements of $S_+$ with nonzero component in the $\b{e}_i^k$ direction are $\b{e}_i^k-\b{e}_j^{k-1}$ for $i<j$ and $\b{e}_j^{k+1}-\b{e}_i^k$ for $i>j$. In particular, there are $N-i$ elements of $S_+$ that have inner product 1 with $\b{e}_i^k$ and $i-1$ elements with inner product $-1$. Thus the inner product $\langle\lambda_I,\b{e}_i^k\rangle$ is bounded above by $\frac{N+1-2i}{2}+(i-1)=\frac{N-1}{2}$ and bounded below by $\frac{N+1-2i}{2}-(N-i)=\frac{-N+1}{2}$. Note that this is independent of $i$, so we obtain that the components of $\lambda_I$ in the directions $\b{e}_1^k,\dots,\b{e}_N^k$ are all in the range $\left[\frac{-N+1}{2},\frac{N-1}{2}\right]$, and they are either all integers or all half-integers. If there are any $i\neq j$ such that the components of $\lambda_I$ in the $\b{e}_i^k,\b{e}_j^k$ directions are equal, then the reflection $s_{i,j}$ in $W_k$ fixes $\lambda_I$. Otherwise, note that there are exactly $N$ (half) integers in the range $\left[\frac{-N+1}{2},\frac{N-1}{2}\right]$, so the components of $\lambda_I$ in the $N$ coordinates $\b{e}_1^k,\dots,\b{e}_N^k$ are exactly some permutation of the components of $\rho_k$, so the $k$-th component of $\lambda_I$ lies in the $W_k$-orbit of $\rho_k$. Doing this for all $k$ yields the claim.
\end{proof}

Note that taking $l=1$, the proof of this case degenerates to the proof for the adjoint case in type $A$.

\subsubsection{Relation to Heaton's formula}\label{sec:heaton}

In \cite[Theorem 2.2]{heaton}, Heaton discovered a formula for the $\q$-character of the cyclic quiver with all dimensions equal to $2$. His formula involves a sum of irreducible characters indexed by certain subsets of $\Z^l\times\Z^l$. We will show that in the case $N=2$, the result from this section can also be expressed as a sum of (signed) irreducible characters indexed by subsets of $\Z^l\times\Z^l$. The subsets obtained via our methods and the methods of Heaton are different, yielding a nontrivial combinatorial identity by equating the two.

Consider the surjection $\C^\times_k\times SL(K_k)\to GL(K_k)$ where $SL(K_k)$ maps identically to $SL(K_k)\subset GL(K_k)$ and $\C^\times_k\cong\C^\times$ sends a complex number to the corresponding scalar matrix. Then $GL(K_k)$-representations are determined by their pullbacks to $\C^\times_k\times SL(K_k)$, so we may consider $G$-representations as representations of $C\times G_0$ where
\[C=\prod_{k\in\Z/l\Z}\C^\times_k,\quad G_0=\prod_{k\in\Z/l\Z}SL(K_k).\]
Irreducible representations of $C$ may be written as $\C_\lambda$ for $\lambda\in\Z^l$, where $\b{z}=(z_1,\dots,z_l)\in C$ acts on $\C_\lambda$ via the scalar $z_1^{\lambda_1}\dots,z_l^{\lambda_l}$. Identifying the weight lattice of $\sl_2$ with $\Z$, dominant weights of $G_0$ may be written as elements of $\Z_{\geq 0}^l$.

Write a dominant weight $\lambda$ of $G$ as $\lambda=\sum_{k\in\Z/l\Z}(\lambda_1^k\b{e}_1^k+\lambda_2^k\b{e}_2^k)$. Then we have $L_\lambda\cong\C_{\b{c}}\boxtimes L_{\lambda'}$ as a $(C\times G_0)$-representation where $c_k=\lambda_1^k+\lambda_2^k,\lambda'_k=\lambda_1^k-\lambda_2^k$.

Fix a multidegree $\d$ in $\q$, and let us compute the character of $\C[\N_V]_{\d}$ via Theorem \ref{thm}. By the above and from the definition of $\P_q$, we can see that all irreducible representations that appear in $\C[\N_V]_{\d}$ have weight $c_\d:=(d_1-d_l,d_2-d_1,\dots,d_l-d_{l-1})$ with respect to $C$. So it suffices to consider weights with respect to $G_0$.

Let $\b{e}_k$ be the fundamental weight of $SL(K_k)$. By abuse of notation, we will denote by $v^k_{i,j}\in V_k$ the weight of $v^k_{i,j}$. Then the restriction to $G_0$ sends
\[v^k_{1,1}\mapsto-\b{e}_k+\b{e}_{k+1},\quad v^k_{2,1}\mapsto-\b{e}_k-\b{e}_{k+1},\quad v^k_{2,2}\mapsto\b{e}_k-\b{e}_{k+1}.\]
For a weakly negative function $m:S_{\leq 0}\to\Z_{\geq 0}$ with $\sum_{\alpha\in S_k}m(\alpha)=d_k$, let 
\[\Lambda(m)=(d_1-2m(v^1_{2,2}),d_2-2m(v^2_{2,2}),\dots,d_l-2m(v^l_{2,2})),\quad\lambda(m)=-\sum_{\alpha\in S}m(\alpha).\]
Then the map $\{\text{weakly negative functions}\}\to\Z^l\times\Z^l$ sending $m\mapsto(\Lambda(m),\lambda(m))$ is injective. Indeed, $m(v^k_{2,2})$ can be read off from $\Lambda(m)$ for all $k$. Then we have
\begin{equation}\label{eq:heaton}
    \lambda(m)_k=d_k+d_{k-1}-2m(v^k_{2,2})-2m(v^{k-1}_{1,1}),
\end{equation}
so $m(v^k_{1,1})$ can be read off from $\lambda(m)$, and finally we can read off the value of $m(v^k_{2,1})$ from ${m(v^k_{1,1})+m(v^k_{2,2})+m(v^k_{2,1})=d_k}$. For $z\in\Z$, let $\bar{z}=z$ if $z\geq 0$ and $\bar{z}=-z-2$ if $z<0$, and let $\sigma(z)=\begin{cases}
    1&z\geq 0\\0&z=-1\\-1&z\leq -2
\end{cases}$. Extend the bar operation to $\Z^l$ by acting coordinate-wise and extend $\sigma$ to $\Z^l$ be acting multiplicatively on coordinates. Then Formula \eqref{eq:formula} reads
\begin{equation}\label{eq:heaton2}
    \ch_\q\C[\N_V]\cong\chi_{c_\d}\left(\sum_{m\text{ weakly negative}}\sigma(\lambda(m))\chi_{\bar{\lambda(m)}}\right).
\end{equation}
Using the fact that $(\Lambda,\lambda)$ is injective, we may identify the sum over $m$ with a sum over a subset of $\Z^l\times\Z^l$. Clearly, $\Lambda(m)_k\in\{d_k,d_k-2,\dots,-d_k\}$. Moreover, since $m$ is weakly negative, we cannot have $m(v^k_{2,2})>0$ for all $k$, so there must be at least one $k$ such that $\Lambda(m)_k=d_k$. Then from \eqref{eq:heaton} we obtain $\lambda(m)_k\in\{\Lambda(m)_k+d_{k-1},\Lambda(m)_k+d_{k-1}-2,\dots,\Lambda(m)_k-\Lambda(m)_{k-1}\}$. Note that $\lambda(m)_k=\Lambda(m)_k+d_{k-1}$ exactly when $m(v^{k-1}_{1,1})=0$, so since $m$ is weakly negative, there must be at least one $k$ with $\lambda(m)_k=\Lambda(m)_k+d_{k-1}$. By the description of $\C[V^{\rho^\vee}]^\h$ in Section \ref{sec:eqdim}, it is clear that these encompass all of the conditions for $m$ to be weakly negative, so we can view weakly negative functions as pairs $(\Lambda(m),\lambda(m))$ satisfying all of the above conditions, and the sum in \eqref{eq:heaton2} can be written as a sum over such pairs.

\subsection{Proof for the cyclic quiver with two vertices}\label{sec:2vertex}

Following the previous case, we prove all of the parts of Proposition \ref{prop:hesselink2} hold for another class of cyclic quivers. Let $K_1,K_2$ be two complex vector spaces of dimensions $d_1,d_2$ which are not necessarily the same. Without loss of generality, let us say $d_1\geq d_2$. Let $G=GL(K_1)\times GL(K_2)$, $V_1=\Hom(K_1,K_2), V_2=\Hom(K_2,K_1)$, and $V=V_{d_1,d_2}=V_1\oplus V_2$. Using the same notation as in the previous section, $\g$ has a basis given by $\{\alpha^k_{i,j}\}_{k\in\{1,2\},1\leq i,j\leq d_k}$ and $\h$ has basis $\{\alpha^k_{i,i}\}$. We have $W=S_{d_1}\times S_{d_2}$. A basis for the weight lattice is given by $\{\b{e}_i^k\}_{k\in\{1,2\},1\leq i\leq d_k}$. Identifying $\h\cong\h^*$ via the dot product, we have $\rho=\rho^\vee=(\rho_1,\rho_2)$ where $\rho_k=\sum_{i=1}^{d_k}\frac{d_k+1-2i}{2}\b{e}_i^k$. As before, a basis of $V$ is given by $\{v_{i,j}^k\}$ where $k\in\{1,2\}$, and if $k=1$ then $1\leq i\leq d_2,1\leq j\leq d_1$ and if $k=2$ then $1\leq i\leq d_1,1\leq j\leq d_2$. Each $v_{i,j}^k$ is a weight vector with weight $\b{e}_i^{k+1}-\b{e}_j^k$. We have

\begin{align*}
    S_+&=\left\{\b{e}_i^{k+1}-\b{e}_j^k\mid d_{k+1}-2i >d_k-2j\right\}\\
    S_-&=\left\{\b{e}_i^{k+1}-\b{e}_j^k\mid d_{k+1}-2i<d_k-2j\right\}\\
    S_0&=\left\{\b{e}_i^{k+1}-\b{e}_j^k\mid d_{k+1}-2i=d_k-2j\right\}.
\end{align*}

In particular, if $d_1\equiv d_2\pmod{2}$ then $S_0$ consists of elements $\b{e}_{j+\frac{d_1-d_2}{2}}^1-\b{e}_j^2,\b{e}_j^2-\b{e}_{j+\frac{d_1-d_2}{2}}^1$ for $1\leq j\leq d_2$ and if $d_1\not\equiv d_2\pmod{2}$ then $S_0=\varnothing$. So we obtain in the former case that $\dim_q\C[V^{\rho^\vee}]^\h=\frac{1}{(1-q_1q_2)^{d_2}}$ and in the latter case that $\dim_q\C[V^{\rho^\vee}]^\h=1$. It is also clear that $\C[V^{\rho^\vee}]$ is free over $\C[V^{\rho^\vee}]^\h$ in both cases.

As in the case of the cyclic quiver with equal dimensions, this case falls under the framework of $\Z/l\Z$-graded Lie algebras, and from \cite{vinberg} we obtain that $V$ is cofree.

\begin{proposition}\label{prop:2vertex}
    Let $V$ be the representation associated to a cyclic quiver with 2 vertices. Then $V$ is Hesselink-type.
\end{proposition}

\begin{proof}
    Let us first assume that $d_1-d_2\geq 2$. We use many of the same ideas as in the proof of Proposition \ref{prop:eqdim}. Let $\lambda_I:=\rho-\sum_{\alpha\in I}\alpha$ for $I\subset S_+$ as in that proof, and note that as before, we only need to consider terms corresponding to subsets $I$ with $\lambda_I$ not on any reflection hyperplane by $W$-antiinvariance. We may also bound the values $\langle\lambda_I,\b{e}_i^k\rangle$ using the same method as in Proposition \ref{prop:eqdim}. We obtain the following bounds:
    
    \begin{align*}
        i\leq\frac{d_1-d_2+1}{2}:\quad&\langle\lambda_I,\b{e}_i^1\rangle\in\Big[\frac{d_1+1-2i-2d_2}{2},\frac{d_1+1-2i}{2}\Big]\\
        \frac{d_1-d_2+1}{2}<i<\frac{d_1+d_2+1}{2}:\quad&\langle\lambda_I,\b{e}_i^1\rangle\in\Big[\frac{-d_2+1}{2},\frac{d_2-1}{2}\Big]\\
        i\geq\frac{d_1+d_2+1}{2}:\quad&\langle\lambda_I,\b{e}_i^1\rangle\in\Big[\frac{d_1+1-2i}{2},\frac{d_1+1-2i+2d_2}{2}\Big].
    \end{align*}
    In particular, $\langle\lambda_I,\b{e}_i^1\rangle$ is always in the range $[\frac{-d_1+1}{2},\frac{d_1-1}{2}]$, so if $\lambda_I$ does not lie on any reflection hyperplane then the set of numbers $\langle\lambda_I,\b{e}_i^1\rangle$ must be exactly the set of (half) integers in the range $[\frac{-d_1+1}{2},\frac{d_1-1}{2}]$. In addition, the only $i$ for which we can have $\langle\lambda_I,\b{e}_i^1\rangle=\frac{d_1-1}{2}$ is $i=1$, and the only $i$ for which we can have $\langle\lambda_I,\b{e}_i^1\rangle=\frac{-d_1+1}{2}$ is $i=d_1$. So we must have these equalities, which is only possible if $\b{e}_1^1-\b{e}_j^2,\b{e}_j^2-\b{e}_{d_1}^1\not\in I$ for all $j$. Embedding the weight data of $V_{d_1-2,d_2}$ inside $V_{d_1,d_2}$ via $\b{e}_i^1\mapsto\b{e}_{i+1}^1,\b{e}_j^2\mapsto\b{e}_j^2$, we thus obtain that the only $\lambda_I$ that do not lie in a reflection hyperplane come from corresponding $\lambda_I$ for $V_{d_1-2,d_2}$. So if $V_{d_1-2,d_2}$ is Hesselink-type then so is $V_{d_1,d_2}$.

    It remains to show that $V_{d_1,d_2}$ is Hesselink-type if $d_1-d_2\leq 1$. If $d_1=d_2$, then this is the 2 vertex case of Proposition \ref{prop:eqdim}. 

    Consider when $d_1=d_2+1$. The same argument as in the $d_1-d_2\geq 2$ case shows that $\langle\lambda_I,\b{e}_i^1\rangle$ is always in the range $\left[\frac{-d_1+1}{2},\frac{d_1-1}{2}\right]$. Applying similar bounds to $\langle\lambda_I,\b{e}_j^2\rangle$ yields that these values are always in the range $[\frac{-d_2-1}{2},\frac{d_2+1}{2}]$. However, if $\langle\lambda_I,\b{e}_j^2\rangle=\frac{d_2+1}{2}$ for some $j$, then we must have $\b{e}_i^1-\b{e}_j^2\in I$ for $i\leq j$ and $\b{e}_j^2-\b{e}_i^1\not\in I$ for $i>j$. In particular, this makes it impossible for there to be some $i$ with $\langle\lambda_I,\b{e}_i^1\rangle=\frac{d_2}{2}$, so these $\lambda_I$ always lie on a reflection hyperplane. Similarly, if there is some $j$ with $\langle\lambda_I,\b{e}_j^2\rangle=\frac{-d_2-1}{2}$, then $\lambda_I$ lies on a reflection hyperplane. So if $\lambda_I$ does not lie on a reflection hyperplane then $\langle\lambda_I,\b{e}_j^2\rangle\in\left[\frac{-d_2+1}{2},\frac{d_2-1}{2}\right]$. In particular, the values $\langle\lambda_I,\b{e}_j^2\rangle$ are exactly the (half) integers in $\left[\frac{-d_2+1}{2},\frac{d_2-1}{2}\right]$, so we obtain that $\lambda_I\in W\rho$ and $V$ is Hesselink-type.
\end{proof}

Thus we obtain Theorem \ref{thm} in this case by Proposition \ref{prop:hesselink2}.

\begin{remark}\label{rem:>2}
    If one tries to replicate this proof for a representation of a cyclic quiver with more than 2 vertices, the bounds on $\langle\lambda_I,\b{e}_i^k\rangle$ from the proof of the $d_1=d_2+1$ case of Proposition \ref{prop:2vertex} do not hold, as you can have contributions to $\langle\lambda_I,\b{e}_i^k\rangle$ from weights involving both the $k-1$-st and $k+1$-st vertices, and these two are no longer the same. For example, for the cyclic quiver with $4$ vertices and dimension vector $(3,2,3,2)$, there exist $I$ for which
    \[\lambda_I=\left(\rho_1,\left(\frac{3}{2},\frac{1}{2}\right),w\rho,\left(-\frac{1}{2},-\frac{3}{2}\right)\right)\]
    where $w$ is any $3$-cycle. So we do not have $\lambda_I\in W\rho$. One can check that the terms corresponding to these $\lambda_I$ do not cancel in the numerator of $X_V$, so \eqref{eq:id} cannot hold as the left hand side does not lie in $\C(q)$.
\end{remark}

\begin{remark}\label{rem:frohmaderheaton}
    In \cite[Theorem 4.3]{frohmaderheaton}, the authors prove a $q$-character formula for all cyclic quivers, but only in degrees at most $\min_k\dim(K_k)$. By comparing our formulas obtained in Sections \ref{sec:eqdim} and \ref{sec:2vertex} with theirs, we obtain nontrivial combinatorial identities in these cases and degree ranges. However, to our knowledge, unlike in Section \ref{sec:heaton}, our formula does not express via the same combinatorial data as their formula. Their formula expresses as a sum over certain $2l$-tuples of semistandard Young tableaux, which are obtained from dominant weights of $G$ via the stable branching rule of \cite[Theorem 2.1.4.1 (a)]{howetanwillenbring}. However, our formula involves nondominant weights as well, so the same cannot be done in our formula.
\end{remark}

\section{Case where $\g$ is a product of copies of $\sl_2$}\label{sec:sl2}

Consider when $\g$ is a (finite) product of copies of $\sl_2$. We will provide a class of representations that exhausts all Hesselink-type representations for $\g$ and we will prove Theorem \ref{thm} for these representations. To this end, we introduce a graphical representation based on quiver representations under which our class of representations falls.

Let $A=(X,Y)$ be an ordered pair where $X$ is a finite set and $Y$ is a finite {\emph{multiset}} consisting of elements of $X$, unordered pairs of elements of $X$, and unordered triples of distinct elements of $X$. For example, one could have $X=\{1,2,3,4\}$ and $Y=\{1,1,\{3,3\},\{2,4\},\{1,3,4\}\}$. We may represent such an $A$ via a diagram, where

\begin{enumerate}
    \item Vertices correspond to elements of $X$.
    \item For each element $a\in Y$ with $a\in X$, we put a framing on the vertex, represented by an edge from the vertex to empty space.
    \item For each pair $\{a,b\}\in Y$ with $a,b\in X$, we draw an edge between $a$ and $b$. If $a=b$, this is a self-loop.
    \item For each triple $\{a,b,c\}\in Y$ with $a,b,c\in X$ pairwise distinct, we draw a filled triangle between $a,b,c$, which we will call a $3$-hyperedge.
\end{enumerate}

We call such a diagram an \textbf{extended quiver}, which we will motivate later. For example, for the $X,Y$ in the previous paragraph, we have the following diagram:

\begin{figure}[H]
    \centering
    \begin{tikzpicture}
        \tikzset{every loop/.style={min distance=10mm, in=-30, out=60}}
        
        \node [below] at (0,0) {$1$};
	\draw [fill] (0,0) circle [radius=0.05];
        \node [left] at (0.5,0.8) {$3$};
	\draw [fill] (0.5,0.8) circle [radius=0.05];
        \node [below] at (1,0) {$4$};
	\draw [fill] (1,0) circle [radius=0.05];
        \node [right] at (2,0) {$2$};
	\draw [fill] (2,0) circle [radius=0.05];
        \draw (1,0)--(2,0);
        \draw (0,0)--(-1,0);
        \draw (0,0)--(-0.5,0.8);
        \draw (0,0)--(0.5,0.8)--(1,0)--(0,0);
        \draw (0.5,0.8) edge [loop above] node {} (0.5,0.8);
        \fill[gray, opacity=0.3] (0,0)--(1,0)--(0.5,0.8)--(0,0);
    \end{tikzpicture}{\caption{an extended quiver}\label{extquiver}}
\end{figure}
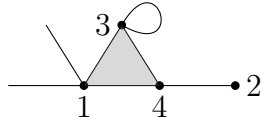

Attached to such an $A$ we also have a representation $V(A)$ of $\prod_{a\in X}(\sl_2)_a$. The irreducible components of $V(A)$ correspond to elements of $Y$, and are as follows:

\begin{enumerate}
    \item For $a\in Y$: the tautological representation $L_1$ of $(\sl_2)_a$,
    \item For $\{a,a\}\in Y$: the representation $L_2$ of $(\sl_2)_a$,
    \item For $\{a,b\}\in Y$ with $a\neq b$: the representation $L_1\boxtimes L_1$ of $(\sl_2)_a\times(\sl_2)_b$,
    \item For $\{a,b,c\}\in Y$: the representation $L_1\boxtimes L_1\boxtimes L_1$ of $(\sl_2)_a\times(\sl_2)_b\times(\sl_2)_c$.
\end{enumerate}
Note that if there are no triples in $Y$, then up to trivial summands, this is a framed quiver representation as seen in the theory of quiver varieties (see \cite[Section 3.1]{ginzburg2009lecturesonnakajimaquivervarieties}), where the dimension of the vector space at each vertex is $2$ and the dimension of the framing at each vertex is the multiplicity of that vertex in $Y$. This is the motivation for the name extended quiver, and we will call the corresponding representations \textbf{extended quiver representations}.

We say an extended quiver is of \textbf{trivial type} if it is one of the following (note the recursive definition):

\begin{enumerate}
    \item A single vertex, a vertex with a framing, a path graph with a single framing on each terminal vertex (a path with $1$ vertex has two framings on the vertex), an $l$-cycle (a $1$-cycle is a self-loop), or a $3$-hyperedge,
    \item An extended quiver of trivial type with a leaf attached to one of its vertices,
    \item A disjoint union of extended quivers of trivial type.
\end{enumerate}

For example, the following extended quiver is of trivial type, while the one in Figure \ref{extquiver} is not.

\begin{figure}[H]
    \centering
    \begin{tikzpicture}
	\draw [fill] (0,0) circle [radius=0.05];
	\draw [fill] (0.5,0.8) circle [radius=0.05];
	\draw [fill] (1,0) circle [radius=0.05];
	\draw [fill] (2,0) circle [radius=0.05];
        \draw [fill] (-1,0) circle [radius=0.05];
        \draw [fill] (3,0) circle [radius=0.05];
        \draw [fill] (2.5,0.8) circle [radius=0.05];
        \draw (1,0)--(2,0)--(3,0);
        \draw (2,0)--(2.5,0.8);
        \draw (0,0)--(-1,0);
        \draw (0,0)--(0.5,0.8)--(1,0)--(0,0);
        \fill[gray, opacity=0.3] (0,0)--(1,0)--(0.5,0.8)--(0,0);

        \draw [fill] (5,0) circle [radius=0.05];
        \draw [fill] (5,1) circle [radius=0.05];
        \draw (4,0)--(6,0);
        \draw (5,0)--(5,1);
    \end{tikzpicture}{\caption{an extended quiver of trivial type}\label{extquivertriv}}
\end{figure}
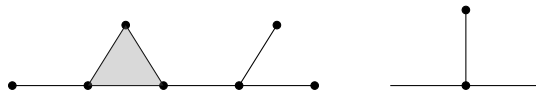

An extended quiver of trivial type arising from the first part of the definition will be called a \textbf{basic extended quiver}. An extended quiver containing no $l$-cycles is called \textbf{acyclic}.

Analogously, we also say an extended quiver representation is of trivial type if its corresponding extended quiver is of trivial type, and we call the representations corresponding to basic extended quivers \textbf{basic extended quiver representations}.

The main result of this section is the following theorem classifying Hesselink-type representations when $\g$ is a product of copies of $\sl_2$, along with the proof of Theorem \ref{thm} in this case.

\begin{theorem}\label{thm:sl2}
    Let $\g=\prod\sl_2$ and let $V$ be a $\g$-representation. 
    \begin{enumerate}
        \item If $V$ is the direct sum of an extended quiver representation of trivial type and any number of copies of the trivial representation, then $V$ is Hesselink-type and Theorem \ref{thm} holds for $V$.\label{thm:sl21}
        \item These exhaust all Hesselink-type $V$.\label{thm:sl22}
    \end{enumerate}
\end{theorem}

The rest of this section will be dedicated to proving this theorem. Note that taking the direct sum of any representation with the trivial representation does not change any of $\C[\N_V]$, $X_V$, or $\P_\q$, so we may ignore trivial summands for the proof of this theorem.

\begin{remark}
    One can use a similar method to check that for $\g=\sl_3$, every Hesselink-type representation satisfies the formula in Theorem \ref{thm}. The only complication is that for $\sl_3$ there is (up to taking duals) one Hesselink-type representation that is not cofree, so additional work is needed to compute $\ch_\q\C[\N_V]$ in this case. This case is elaborated on in Example \ref{eg:noncofree}.
\end{remark}

\subsection{Proof of Theorem \ref{thm:sl2} \ref{thm:sl21}.}

We first prove that Theorem \ref{thm} holds for every extended quiver representation of trivial type. We prove that extended quiver representations of trivial type are all Hesselink-type along the way.

\begin{lemma}\label{lem:cofree}
    Let $A,A_1,A_2$ be extended quivers of trivial type. Let $A'$ be obtained from $A$ by adding a leaf.
    
    \begin{enumerate}
        \item All basic extended quiver representations are cofree.\label{lem:cofree1}
        \item If $V(A)$ is cofree, then $V(A')$ is cofree.\label{lem:cofree2}
        \item If $V(A_1),V(A_2)$ are cofree, then $V(A_1\sqcup A_2)$ is cofree.\label{lem:cofree3}
    \end{enumerate}
\end{lemma}

\begin{proof}
We prove \ref{lem:cofree2}.~first.

Write $V=V(A),V'=V(A')$, and let $V''\cong L_1\boxtimes L_1$ be the representation corresponding to the edge connecting the leaf to $A$, so that $V'=V\oplus V''$. Let $\g'\cong\sl_2$ be the copy of $\sl_2$ corresponding to the leaf. A simple calculation of the dimensions of weight spaces of $\C[V'']_d$ shows that we have
\[\C[V'']_d\cong L_d\boxtimes L_d\oplus L_{d-2}\boxtimes L_{d-2}\oplus L_{d-4}\boxtimes L_{d-4}\oplus\cdots.\]
In particular, irreducible representations in $\C[V'']$ are all of the form $L_i\boxtimes L_i$ for some $i\geq 0$. Now, note that we have a decomposition $\C[V']=\C[V]\otimes\C[V'']$, so any copy of the trivial representation inside $\C[V']$ must be a direct summand of $M\otimes (L_i\boxtimes L_i)$ for some irreducible representation $M$ in $\C[V]$. But since $\g'$ acts trivially on $V$, the action of $\g'$ on $M\otimes (L_i\boxtimes L_i)$ will always be via a direct sum of copies of $L_i$. In particular, for there to be a trivial representation inside $M\otimes (L_i\boxtimes L_i)$ we must have $i=0$. But then we have $M\otimes (L_i\boxtimes L_i)\cong M\otimes\C\cong M$, so we must also have $M\cong\C$. It follows that we have $\C[V']^G=\C[V]^G\otimes\C[V'']^G$.

Note that from the decomposition of $\C[V'']_d$ above we see that $\C[V'']^G$ is a polynomial algebra in one variable in degree 2, so in particular is a PID. Then as a graded torsion-free module with finite-dimensional graded components over a PID, $\C[V'']$ is free, hence $V''$ is cofree. So $\C[V]$ and $\C[V'']$ are free over $\C[V]^G,\C[V'']^G$, respectively, so it follows that $\C[V']=\C[V]\otimes\C[V'']$ is free over $\C[V]^G\otimes\C[V'']^G=\C[V']^G$, so $V'$ is cofree.

\ref{lem:cofree1}.

The single vertex corresponds to the zero representation of $\sl_2$, for which we have $\C[V]^G=\C[V]$. So this is cofree. A vertex with a framing corresponds to the representation $V=L_1$ of $\sl_2$, and it is well-known that $\C[V]$ decomposes as $\C\oplus L_1\oplus L_2\oplus\cdots$ as a representation of $\sl_2$, so $\C[V]^G=\C$ and $V$ is cofree.

Now, consider the path on $l$ vertices with framings on each terminal vertex. Let $(L_i)_k$ be the representation on which the $k$-th copy of $\sl_2$ acts by $L_i$ and all other copies of $\sl_2$ act trivially. Then from the proof of Lemma \ref{lem:cofree} \ref{lem:cofree2}.~we have that
\begin{equation}\label{eq:freegen}
    \ch_q\C[(L_1)_k\otimes(L_1)_{k+1}]=\frac{1}{1-q^2}\sum_{d\geq 0}\ch((L_d)_k\otimes(L_d)_{k+1})q^d.
\end{equation}
On the other hand we have $\ch_q\C[(L_1)_k]=\sum_{d\geq 0}\ch(L_d)_kq^d$. Note that $V$ is the direct sum of the representations $(L_1)_1,(L_1)_1\otimes (L_1)_2,\dots,(L_1)_{l-1}\otimes (L_1)_l,(L_1)_l$. Then it is clear that for each $d\geq 0$, there is a copy of $\C$ in $\C[V]$ for each representation of the form
\[(L_d)_1\otimes((L_d)_1\otimes(L_d)_2)\otimes\cdots\otimes((L_d)_{l-1}\otimes (L_d)_l)\otimes (L_d)_l\]
inside $\C[(L_1)_1]\otimes\C[(L_1)_1\otimes (L_1)_2]\otimes\cdots$, and that these are the only copies of $\C$ that are in $\C[V]$. It follows that $\C[V]^G$ is freely generated by the generators of $\C[(L_1)_k\otimes(L_1)_{k+1}]^G$ for $1\leq k\leq l-1$ and the lowest degree copy of $\C$ in
\[(L_1)_1\otimes((L_1)_1\otimes(L_1)_2)\otimes\cdots\otimes((L_1)_{l-1}\otimes (L_1)_l)\otimes (L_1)_l.\]
Let $f$ be this last generator.

From \eqref{eq:freegen} we obtain that a set of free generators of $\C[V]$ over the subalgebra of $\C[V]^G$ generated by $\C[(L_1)_k\otimes(L_1)_{k+1}]^G$ for $1\leq k\leq l-1$ is given by bases of representations of the form
\[(L_{d_0})_1\otimes((L_{d_1})_1\otimes(L_{d_1})_2)\otimes\cdots\otimes((L_{d_{l-1}})_{l-1}\otimes (L_{d_{l-1}})_l)\otimes (L_{d_l})_l\]
where each of the tensor factors $(L_{d_k})_k\otimes(L_{d_k})_{k+1}$ is the lowest degree occurrence of the factor in $\C[(L_1)_k\otimes(L_1)_{k+1}]$. Let the space generated by these generators be $M$. Then $M$ itself is a module over $\C[f]$, as multiplication by $f$ sends a summand of 
\[(L_{d_0})_1\otimes((L_{d_1})_1\otimes(L_{d_1})_2)\otimes\cdots\otimes((L_{d_{l-1}})_{l-1}\otimes (L_{d_{l-1}})_l)\otimes (L_{d_l})_l\]
to a summand of
\[(L_{d_0+1})_1\otimes((L_{d_1+1})_1\otimes(L_{d_1+1})_2)\otimes\cdots\otimes((L_{d_{l-1}+1})_{l-1}\otimes (L_{d_{l-1}+1})_l)\otimes (L_{d_l+1})_l\]
in degree one higher in every multigrading variable. Thus as a torsion-free graded module over the PID $\C[f]$, $M$ is a free $\C[f]$-module by graded Nakayama's lemma. So combining this with the fact that $M$ has a basis of free generators of $\C[V]$ over the other generators of $\C[V]^G$, we obtain that $V$ is cofree.

For the $l$-cycle with $l>1$, we have
\[V=\bigoplus_{k=1}^l(L_1)_k\boxtimes(L_1)_{k+1}\]
where we define $(L_1)_{l+1}=(L_1)_1$. A similar computation to the previous case shows that $\C[V]^G$ is freely generated by the generators of $\C[(L_d)_k\otimes(L_d)_{k+1}]^G$ and the lowest copy of $\C$ in
\[((L_1)_1\otimes(L_1)_2)\otimes\cdots\otimes((L_1)_{l-1}\otimes (L_1)_l)\otimes((L_1)_l\otimes(L_1)_1).\]
Then an analogous argument to the previous case implies that $V$ is cofree.

For the $3$-hyperedge, it was proved in \cite{littelmann} that $V$ is cofree.

\ref{lem:cofree3}.

Let $V_1=V(A_1),V_2=V(A_2),V=V_1\oplus V_2$. Let $\g_1,\g_2$ be the direct sum of all copies of $\sl_2$ corresponding to vertices in $A_1,A_2$, respectively. Then any copy of the trivial representation inside $\C[V]$ is a direct summand of $M_1\otimes M_2$ for some irreducible representations $M_1,M_2$ in $\C[V_1],\C[V_2]$, respectively. But $\g_1$ acts trivially on $M_2$, so if it acts trivially on some summand of $M_1\otimes M_2$ then it must act trivially on $M_1$. Similarly, $\g_2$ must act trivially on $M_2$. It follows that $M_1\cong M_2\cong\C$, so we have $\C[V]^G=\C[V_1]^G\boxtimes\C[V_2]^G$. Then $\C[V]=\C[V_1]\boxtimes\C[V_2]$ is the tensor product of two free modules, hence is free over $\C[V]^G$.
\end{proof}

\begin{lemma}\label{lem:trivtype}
    Let $A,A_1,A_2$ be extended quivers of trivial type. Let $A'$ be obtained from $A$ by adding a leaf.
    
    \begin{enumerate}
        \item All basic extended quiver representations are Hesselink-type.\label{lem:trivtype1}
        \item If $V(A)$ is Hesselink-type, then $V(A')$ is Hesselink-type.\label{lem:trivtype2}
        \item If $V(A_1),V(A_2)$ are Hesselink-type, then $V(A_1\sqcup A_2)$ is Hesselink-type.\label{lem:trivtype3}
    \end{enumerate}
\end{lemma}

\begin{proof}
\ref{lem:trivtype1}.

A single vertex corresponds to the zero representation of $\sl_2$, for which we have $X_V=1$.

A vertex with framing corresponds to $V=L_1$, for which we have $S_+=\{1\}$. Borrowing notation from Section \ref{sec:eqdim}, we have $\lambda_{\{1\}}=0$, which lies on a reflection hyperplane. So we are left with $X_V=1$.

Let $V$ be a path on $l$ vertices with framings on each terminal vertex. For the $k$-th copy of $\sl_2$, let its fundamental weight be $\b{e}_k$. Then $S_+$ consists of $\b{e}_1,\b{e}_l$, and $\b{e}_k+\b{e}_{k+1}$ for $1\leq k\leq l-1$. From this description it is also clear that the only $I\neq\varnothing\subset S_+$ for which $\lambda_I$ does not lie on a reflection hyperplane is $I=S_+$, for which we have $\lambda_I=-\sum_{k=1}^l\b{e}_k$. This lies in $W\rho$, so $V$ is Hesselink-type.

The $1$-cycle corresponds to the adjoint representation of $\sl_2$, for which we already know \eqref{eq:id} holds.

For the $l$-cycle with $l>1$, we have $S_+=\{\b{e}_k+\b{e}_{k+1}\}$ for $k\in\Z/l\Z$, and again it is easy to see that the only $I\neq\varnothing\subset S_+$ for which $\lambda_I$ does not lie on a reflection hyperplane is $I=S_+$, for which we have $\lambda_I=-\sum_{k\in\Z/l\Z}\b{e}_k$. 

Finally, for the $3$-hyperedge, we have $S_+=\{\b{e}_1+\b{e}_2+\b{e}_3,-\b{e}_1+\b{e}_2+\b{e}_3,\b{e}_1-\b{e}_2+\b{e}_3,\b{e}_1+\b{e}_2-\b{e}_3\}$. The $I\neq \varnothing\subset S_+$ for which $\lambda_I$ does not lie on a reflection hyperplane are all 2 and 4 element subsets of $S_+$. A short computation shows that the 2 element subsets cancel out in $X_V$, and the 4 element subset yields $\lambda_I=-\rho$.

\ref{lem:trivtype2}.

Let $V=V(A),V'=V(A')$. Let $V''\cong L_1\boxtimes L_1$ be the summand of $V'$ corresponding to the edge connecting the leaf, and let $q''$ be the variable $q$ corresponding to $V''$. From the proof of Lemma \ref{lem:cofree} \ref{lem:cofree2}.~we have that 
\begin{equation*}
\dim_\q\C[V']^G=\dim_\q\C[V]^G\dim_\q\C[V'']^G=\frac{\dim_\q\C[V]^G}{(1-q''^2)}.
\end{equation*}
Let $\g''\cong\sl_2$ be the copy of $\sl_2$ in the Lie algebra for $A'$ corresponding to the leaf and let $\omega$ be the fundamental weight of $\g''$. We have $S_0(V')=S_0(V)\cup S_0(V'')$. Note that every element of $S_0(V)$ is orthogonal to $\omega$, so the only functions in $(S_0)_0$ must be equal on both elements of $S_0(V'')$. This implies $\C[(V')^{\rho^\vee}]^\h=\C[V^{\rho^\vee}]^\h\otimes\C[(V'')^{\rho^\vee}]^\h$, so $\dim_\q\C[(V')^{\rho^\vee}]^\h=\dim_\q\C[V^{\rho^\vee}]^\h\dim_\q\C[(V'')^{\rho^\vee}]^\h=\frac{\dim_\q\C[V^{\rho^\vee}]^\h}{1-q''^2}$. So we have $\frac{\dim_\q\C[(V')^{\rho^\vee}]^\h}{\dim_\q\C[V']^G}=\frac{\dim_\q\C[V^{\rho^\vee}]^\h}{\dim_\q\C[V]^G}$.

We have $S_+(V')=S_+(V)\cup S_+(V'')$, and $S_+(V'')$ consists of a single element, call it $\alpha'$. Then note that for any $I\subset S_+$, if $\alpha'\in I$ then $\langle\lambda_I,\omega\rangle=0$, so $\lambda_I$ lies on a reflection hyperplane. So the only $I$ that contribute to $X_{V'}$ are subsets of $S_+(V)$, and we get $X_V=X_{V'}$. So \eqref{eq:id} holds for $V'$ since it holds for $V$.

\ref{lem:trivtype3}.

Let $V'=V(A_1),V''=V(A_2),V=V(A_1\sqcup A_2)$. Then we have $\C[V]=\C[V']\boxtimes\C[V'']$, so we have $\dim_\q\C[V]^G=\dim_\q\C[V']^G\dim_\q\C[V'']^G$. We also have $S_+(V)=S_+(V')\cup S_+(V''),S_0(V)=S_0(V')\cup S_0(V'')$. Moreover, any element of $S_+(V')$ is always orthogonal to any element of $S_+(V'')$, and the same holds for $S_0(V'),S_0(V'')$. The former implies $X_V=X_{V'}X_{V''}$ and the latter implies $\dim_\q\C[V^{\rho^\vee}]^\h=\dim_\q\C[(V')^{\rho^\vee}]^\h\dim_\q\C[(V'')^{\rho^\vee}]^\h$. So \eqref{eq:id} holds for $V$ since it holds for $V',V''$.
    
\end{proof}

Now, to prove Theorem \ref{thm} for extended quiver representations, we first reduce to basic extended quiver representations.

\begin{lemma}\label{lem:leafdisjoint}
    Let $A,A_1,A_2$ be extended quivers of trivial type. Let $A'$ be obtained from $A$ by adding a leaf.
    \begin{enumerate}
        \item If Theorem \ref{thm} holds for $V(A)$ then it also holds for $V(A')$.\label{lem:leafdisjoint1}
        \item If Theorem \ref{thm} holds for $V(A_1)$ and $V(A_2)$ then it also holds for $V(A_1\sqcup A_2)$.\label{lem:leafdisjoint2}
    \end{enumerate}
\end{lemma}

\begin{proof}
    \ref{lem:leafdisjoint1}.

        We reuse notation from the proof of Lemma \ref{lem:cofree} \ref{lem:cofree2}. Let $k,k'$ denote the vertex attached to the leaf and the leaf itself, respectively, and let $\b{e}_k,\b{e}_{k'}$ be the corresponding fundamental weights. Then we have $S(V'')=\{\pm\b{e}_k\pm\b{e}_{k'}\}$, and we have $S_0(V'')=\{\b{e}_k-\b{e}_{k'},\b{e}_{k'}-\b{e}_k\},S_-(V'')=\{-\b{e}_k-\b{e}_{k'}\}$. 
        
        Recall from the proof of Lemma \ref{lem:cofree} \ref{lem:cofree2}.~that we have $\C[V']=\C[V]\otimes\C[V'']$, $\C[V']^G=\C[V]^G\otimes\C[V'']^G$, so we have $\C[\N_{V'}]\cong\C[\N_V]\otimes\C[\N_{V''}]$. From there we also obtain $\dim_\q\C[V'']^G=\frac{1}{1-q''^2}$, so we have
        \[\ch_\q\C[\N_{V'}]=\ch_\q\C[\N_V]\left(\frac{1-q''^2}{\prod_{a_k,a_{k'}=\pm1}(1-q''t^{a_k\b{e}_k+a_{k'}\b{e}_{k'}})}\right).\]
        Let
        \[Q_V:=\sum_{\lambda\in\Lambda}\P_{\q,V}t^\lambda\]
        be the generating series for $\P_{\q,V}$, and define $Q_{V'},Q_{V''}$ similarly. Then we have $Q_{V'}=Q_VQ_{V''}$, and Theorem \ref{thm} is equivalent to the identity
        \[\ch_\q[\N_{V'}]=\frac{\sum_{w\in W}(-1)^{\ell(w)}t^{w\rho}w(Q_{V'})}{\sum_{w\in W}(-1)^{\ell(w)}t^{w\rho}}=\frac{\sum_{w\in W}(-1)^{\ell(w)}t^{w\rho}w(Q_VQ_{V''})}{\sum_{w\in W}(-1)^{\ell(w)}t^{w\rho}}.\]
        Let us now compute $Q_{V''}$. A weakly negative function is zero on at least one of the two elements of $S_0(V'')$. So we obtain
        \[Q_{V''}=\frac{1}{1-q''t^{\b{e}_k+\b{e}_{k'}}}\left(\frac{1}{1-q''t^{\b{e}_k-\b{e}_{k'}}}+\frac{1}{1-q''t^{\b{e}_{k'}-\b{e}_k}}-1\right).\]
        Let $W_{k'}$ be the Weyl group of the $k'$ copy of $\sl_2$, and let $s$ be the nontrivial element of $W_{k'}$. Then direct computation shows that
        \[\frac{(t^{\b{e}_{k'}}Q_{V''})-s(t^{\b{e}_{k'}}Q_{V''})}{t^{\b{e}_{k'}}-t^{-\b{e}_{k'}}}=\frac{1-q''^2}{\prod_{a_k,a_{k'}=\pm1}(1-q''t^{a_k\b{e}_k+a_{k'}\b{e}_{k'}})}.\]
        In particular, the right hand side is $W$-invariant. Let $W_0$ be the subgroup of $W$ corresponding to all vertices other than $k'$. Since $Q_V$ is invariant under $W_{k'}$, we obtain
        \begin{align*}
            \frac{\sum_{w\in W}(-1)^{\ell(w)}t^{w\rho}w(Q_VQ_{V''})}{\sum_{w\in W}(-1)^{\ell(w)}t^{w\rho}}&=\frac{\sum_{w\in W_0}(-1)^{\ell(w)}t^{w(\rho-\b{e}_{k'})}w\left(Q_V\left(\frac{(t^{\b{e}_{k'}}Q_{V''})-s(t^{\b{e}_{k'}}Q_{V''})}{t^{\b{e}_{k'}}-t^{-\b{e}_{k'}}}\right)\right)}{\sum_{w\in W_0}(-1)^{\ell(w)}t^{w(\rho-\b{e}_{k'})}}\\
            &=\frac{\sum_{w\in W_0}(-1)^{\ell(w)}t^{w(\rho-\b{e}_{k'})}w\left(Q_V\right)}{\sum_{w\in W_0}(-1)^{\ell(w)}t^{w(\rho-\b{e}_{k'})}}\left(\frac{1-q''^2}{\prod_{a_k,a_{k'}=\pm1}(1-q''t^{a_k\b{e}_k+a_{k'}\b{e}_{k'}})}\right)
        \end{align*}
        where the first equality is via the Weyl denominator formula and the second equality uses $W$-invariance. Now, the result follows since the left factor in the above is exactly $\ch_\q\C[\N_V]$ by assumption.

    \ref{lem:leafdisjoint2}.

    Let $V=V(A_1\sqcup A_2)$. Using the same idea as in the previous part, both $\ch_\q\C[\N_V]$ and $Q_V$ factor into the respective versions for $V(A_1),V(A_2)$, so the result is obtained simply by taking the product of the respective identities for Theorem \ref{thm} for the two components.
    
\end{proof}

Thus we reduce the proof of Theorem \ref{thm} in this case to the case of basic extended quiver representations. We first deal with the acyclic ones, as the cycle will require more work.

\begin{lemma}\label{lem:freeness}
    Let $V$ be an acyclic basic extended quiver representation. Then $\C[V^{\rho^\vee}]$ is free over $\C[V^{\rho^\vee}]^\h$.
\end{lemma}

\begin{proof}
    For a single vertex, we have $\C[V^{\rho^\vee}]\cong\C[V^{\rho^\vee}]^\h\cong\C$.

    For a vertex with a single framing and a $3$-hyperedge, we have $S_0=\varnothing$. So we have $\C[V^{\rho^\vee}]\cong\C[V^{\rho^\vee}]^\h\cong\C$ in these cases as well.

    Let $V=V(A)$ for $A$ a path on $l$ vertices with framings on each terminal vertex. For the $k$-th copy of $\sl_2$, let its fundamental weight be $\b{e}_k$. Then $S_0$ consists of $\b{e}_k-\b{e}_{k+1},-\b{e}_k+\b{e}_{k+1}$ for $1\leq k\leq l-1$. We then see that $(S_0)_0$ consists of functions on $S_0$ that take the same values on $\b{e}_k-\b{e}_{k+1},-\b{e}_k+\b{e}_{k+1}$ for each $k$, so a set of free generators of $S_0$ over $(S_0)_0$ is given by monomials that have degree 0 in at least one of $\b{e}_k-\b{e}_{k+1},-\b{e}_k+\b{e}_{k+1}$ for each $k$.
\end{proof}

Thus combining Proposition \ref{prop:hesselink2} and Lemmas \ref{lem:cofree}, \ref{lem:trivtype}, \ref{lem:leafdisjoint}, and \ref{lem:freeness}, we obtain

\begin{corollary}\label{prop:sl2s}
    Theorem \ref{thm} holds for all acyclic extended quiver representations of trivial type.
\end{corollary}

Lastly, we prove that Theorem \ref{thm} holds for the $l$-cycle, which completes the proof of Theorem \ref{thm}. Note in particular that $\C[V^{\rho^\vee}]$ is not free over $\C[V^{\rho^\vee}]^\h$ in this case, so we cannot appeal to Proposition \ref{prop:hesselink2}.

\begin{proposition}\label{prop:cycle}
    Theorem \ref{thm} holds for the $l$-cycle.
\end{proposition}

\begin{proof}
    When $l=1$, this is the adjoint case for $\sl_2$, so we will assume $l>1$.
    
    We will reuse notation from the proof of Lemma \ref{lem:trivtype} and the definition of $Q_V$ from the proof of Lemma \ref{lem:leafdisjoint}. We have by definition that
    \[Q_V=\frac{1}{\prod_{\alpha\in S_-}(1-q_\alpha t^{-\alpha})}\ch_q\left(\C[V^{\rho^\vee}]/\C[V^{\rho^\vee}]^\h_+\C[V^{\rho^\vee}]\right).\]
    Write
    \[F:=\prod_{k\in\Z/l\Z}(1-q_kt^{-\b{e}_k-\b{e}_{k+1}}),\quad U_+:=\prod_{k\in\Z/l\Z}(1-q_kt^{\b{e}_k-\b{e}_{k+1}}),\quad U_-:=\prod_{k\in\Z/l\Z}(1-q_kt^{-\b{e}_k+\b{e}_{k+1}}).\]
    Recall that $\C[V^{\rho^\vee}]^\h$ is generated by $p_k:=v_{\b{e}_k-\b{e}_{k+1}}v_{-\b{e}_k+\b{e}_{k+1}}$, $p_+:=\prod_{k\in\Z/l\Z}v_{\b{e}_k-\b{e}_{k+1}}$, and $p_-:=\prod_{k\in\Z/l\Z}v_{-\b{e}_k+\b{e}_{k+1}}$, with the relation $p_+p_-=\prod_{k\in\Z/l\Z}p_k$. Then we have
    \[\ch_\q\left(\C[V^{\rho^\vee}]/\C[V^{\rho^\vee}]^\h_+\C[V^{\rho^\vee}]\right)=\frac{\prod_{k\in\Z/l\Z}(1-q_k^2)-\left(\prod_{k\in\Z/l\Z}q_k\right)(U_++U_-)}{U_+U_-}.\]
    Indeed, first exclude monomials containing $p_k$ for any $k$, and then exclude monomials containing either all of the $v_{\b{e}_k-\b{e}_{k+1}}$ and none of the $v_{-\b{e}_k+\b{e}_{k+1}}$ or vice versa. After rearranging terms, we obtain
    \[Q_V=\ch_\q\C[V]\left(F\left(\prod_{k\in\Z/l\Z}(1-q_k^2)-\left(\prod_{k\in\Z/l\Z}q_k\right)(U_++U_-)\right)\right).\]
    Thus we wish to show that
    \begin{multline}\label{eq:cycle}
        \ch_\q\C[V]\Bigg(\prod_{k\in\Z/l\Z}(1-q_k^2)\frac{\sum_{w\in W}(-1)^{\ell(w)}t^{w\rho}wF}{\sum_{w\in W}(-1)^{\ell(w)}t^{w\rho}}\\-\left(\prod_{k\in\Z/l\Z}q_k\right)\left(\frac{\sum_{w\in W}(-1)^{\ell(w)}t^{w\rho}w(F(U_++U_-))}{\sum_{w\in W}(-1)^{\ell(w)}t^{w\rho}}\right)\Bigg)
    \end{multline}
    is exactly $\ch_\q\C[\N_V]$.
    
    Let us now compute each of the interior sums. We have $t^\rho F=\sum_{I\subset S_+}(-\q)^{|\lambda|}t^{\lambda_I}$ where we reuse the notation $\lambda_I:=\rho-\sum_{\alpha\in I}\alpha$ from Proposition \ref{prop:eqdim}. Then noting that $\rho=\sum_{k\in\Z/l\Z}\b{e}_k$, it is easy to see that $\lambda_I$ will have at least one coordinate equal to zero unless $I=\varnothing$ or $I=S_+$. If there is a coordinate that is zero, then $\lambda_I$ lies on a reflection hyperplane and does not contribute to $\frac{\sum_{w\in W}(-1)^{\ell(w)}t^{w\rho}wF}{\sum_{w\in W}(-1)^{\ell(w)}t^{w\rho}}$. $I=\varnothing$ clearly contributes $1$ to the sum, while $I=S_+$ yields $\lambda_I=-\rho$, so it contributes $(-1)^l\prod_{k\in\Z/l\Z}(-q_k)=\prod_{k\in\Z/l\Z}q_k$ to the sum. So we obtain ${\frac{\sum_{w\in W}(-1)^{\ell(w)}t^{w\rho}wF}{\sum_{w\in W}(-1)^{\ell(w)}t^{w\rho}}=1+\prod_{k\in\Z/l\Z}q_k}$.

    Now, let us compute $\frac{\sum_{w\in W}(-1)^{\ell(w)}t^{w\rho}w(FU_+)}{\sum_{w\in W}(-1)^{\ell(w)}t^{w\rho}}$. We can expand
    \[FU_+=\prod_{k\in\Z/l\Z}(1-q_kt^{-\b{e}_k-\b{e}_{k+1}})(1-q_kt^{\b{e}_k-\b{e}_{k+1}}))=\prod_{k\in\Z/l\Z}(1-q_kt^{-\b{e}_{k+1}}(t^{\b{e}_k}+t^{-\b{e}_k})+q_k^2t^{-2\b{e}_{k+1}}).\]
    Observe that the only nontrivial $W$-action on the $k$-th factor is via its action on $\b{e}_{k+1}$. Thus the numerator $\sum_{w\in W}(-1)^{\ell(w)}t^{w\rho}w(FU_+)$ factors into one-variable sums as
    \[\prod_{k\in\Z/l\Z}\left(\sum_{\epsilon=\pm1}\epsilon t^{\epsilon\b{e}_{k+1}}(1-q_kt^{-\epsilon\b{e}_{k+1}}(t^{\b{e}_k}+t^{-\b{e}_k})+q_k^2t^{-2\epsilon\b{e}_{k+1}})\right)=\prod_{k\in\Z/l\Z}\left((1-q_k^2)(t^{\b{e}_{k+1}}-t^{-\b{e}_{k+1}})\right).\]
    So by the Weyl denominator formula we obtain $\frac{\sum_{w\in W}(-1)^{\ell(w)}t^{w\rho}w(FU_+)}{\sum_{w\in W}(-1)^{\ell(w)}t^{w\rho}}=\prod_{k\in\Z/l\Z}(1-q_k^2)$. A similar proof also shows $\frac{\sum_{w\in W}(-1)^{\ell(w)}t^{w\rho}w(FU_-)}{\sum_{w\in W}(-1)^{\ell(w)}t^{w\rho}}=\prod_{k\in\Z/l\Z}(1-q_k^2)$.
    So \eqref{eq:cycle} is equal to
    \[\ch_\q\C[V]\left(\prod(1-q_k^2)(1+\prod q_k)-2\prod (q_k(1-q_k^2))\right)=\ch_\q\C[V]\prod(1-q_k^2)(1-\prod q_k).\]
    It follows from the proof of Lemma \ref{lem:cofree} that $\ch_\q\C[V]^G=\frac{1}{\prod(1-q_k^2)(1-\prod q_k)}$, so this is exactly $\ch_\q\C[\N_V]$.
\end{proof}

\begin{remark}
    If instead of looking at a product of $\sl_2$'s we look at a product of $\gl_2$'s, then we may interpret the $l$-cycle as the representation of a cyclic quiver with equal dimensions, from which we can obtain the original extended quiver representation for $\prod\sl_2$ by pulling back along the inclusion $\prod\sl_2\hookrightarrow\prod\gl_2$. Then the formula in Theorem \ref{thm} holds for the $\prod\gl_2$-representations by Section \ref{sec:eqdim}. Indeed, for such representations of a product of $\gl_2$'s we do have the condition that $\C[V^{\rho^\vee}]$ is free over $\C[V^{\rho^\vee}]^\h$, because there are fewer $\h$-invariants for $\prod\gl_2$ than for $\prod\sl_2$.

    In principle, one could prove an analog of Theorem \ref{thm:sl2} for a product of $\gl_2$'s by creating analogs of all of the extended quiver representations for $\gl_2$. It would be interesting to see whether such an analog could be used to classify representations for $\prod\gl_2$ such that the formula in Theorem \ref{thm} holds, or whether there are many more such representations.
\end{remark}

\subsection{Proof of Theorem \ref{thm:sl2} \ref{thm:sl22}.}\label{sec:sl22}

In all statements in this section we will ignore trivial summands.

We will now prove that the only Hesselink-type representations for $\g=\prod\sl_2$ are the direct sum of an extended quiver representation of trivial type and copies of the trivial representation. In fact, we will prove a non-multigraded version of this statement, that if 
\[\frac{\sum_{w\in W}(-1)^{\ell(w)}t^{w\rho}\prod_{\alpha\in S_+}(1-qt^{-w\alpha})}{\sum_{w\in W}(-1)^{\ell(w)}t^{w\rho}}\]
is a $\Z[q]$-multiple of the trivial character, then $V$ is an extended quiver representation of trivial type. Note that this is a stronger statement than is required, since the above expression is the restriction of $X_V$ to non-multigraded setting, and if $X_V$ is a $\Z[\q]$-multiple of the trivial character then its restriction is a $\Z[q]$-multiple of the trivial character.

Let $\g=\sl_2^{\oplus l}$ and let $V$ be a $\g$-representation. The unique top degree term in the expression $t^\rho\prod_{\alpha\in S_+}(1-qt^{-\alpha})$ in the numerator of $X_V$ is given by $(-q)^{|S_+|}t^{\rho-\sum_{\alpha\in S_+}\alpha}$, and if $V$ is Hesselink-type then this must lie in either $W\rho$ or a reflection hyperplane. Let $\alpha_k$ be the $k$-th coordinate of $\alpha$. Then if $V$ is Hesselink-type, either there is some $k$ with $1-\sum_{\alpha\in S_+}\alpha_k=0$ or $1-\sum_{\alpha\in S_+}\alpha_k$ is in $\{1,-1\}$ for all $k$.

To this end, define
\[c_k(V)=\sum_{\alpha\in S_+}\alpha_k.\]
Note that $c_k$ behaves additively with respect to direct sums.

\begin{lemma}\label{lem:cbound}
    Let $V=L_{d_1}\boxtimes\cdots\boxtimes L_{d_l}$ be irreducible, and fix $k\in\{1,\dots,l\}$.
    \begin{enumerate}
        \item We have $c_k(V)\geq 0$, with equality exactly when $d_k=0$.\label{lem:cbound1}
        \item We have $c_k(V)=1$ if and only if $d_k=1$ and $d_{k'}>0$ for at most one $k'\neq k$.\label{lem:cbound2}
        \item If $c_k(V)\leq 2$ for all $k$, then $V$ is one of
        \[L_1,\quad L_2,\quad L_1\boxtimes L_1,\quad L_1\boxtimes L_1\boxtimes L_1.\]\label{lem:cbound3}
    \end{enumerate}
\end{lemma}

\begin{proof}
    Let $\alpha\in S(V)$, and fix the values of $\alpha_{k'}$ for all $k'\neq k$. Let $s=\sum_{k'\neq k}\alpha_{k'}$. Then $\alpha\in S_+$ if and only if $\alpha_k+s>0$. So the contribution of all such $\alpha$ to $c_k(V)$ is
    \[\sum_{\substack{\alpha_k\in\{-d_k,-d_k+2,\dots,d_k\}\\\alpha_k+s>0}}\alpha_k=a(d_k-a+1).\]
    where the sum consists of $d_k,d_k-2,\dots,d_k-2(a-1)$. This is always nonnegative since $a\leq d_k+1$. Summing across all $\alpha_{k'}$, this proves the nonnegativity part of \ref{lem:cbound1}.

    We have $c_k(V)=0$ if and only if each of the $a(d_k-a+1)$ is zero for all $s$. Note that the possible values of $s$ are exactly $-\sum_{k'\neq k}d_{k'},-\sum_{k'\neq k}d_{k'}+2,\dots,\sum_{k'\neq k}d_{k'}$. So either $0$ or $1$ is always a possible value of $s$, and if $d_k>0$ then this yields a positive value for $a(d_k-a+1)$. This proves \ref{lem:cbound1}.

    We have $a(d_k-a+1)=1$ only when $d_k=a=1$, so we have $d_k=1$. For $d_k=1$ we have $a=1$ exactly when $s=0,1$. So we have $c_k(V)=1$ if and only if $d_k=1$ and there is exactly one choice of $\alpha_{k'}$ that yields $s=0,1$. By Clebsch-Gordan, this occurs exactly when $d_{k'}>0$ for at most one $k'\neq k$. This proves \ref{lem:cbound2}.

    Finally, we prove part \ref{lem:cbound3}. If there is some $k$ with $d_k\geq 3$, then when $s$ is either $0$ or $1$, we get a contribution of at least $3+1>2$ to $c_k(V)$. If $d_k=2$ and $d_{k'}>0$ for some other $k'$, then $s$ can either be both of $0,2$ or both of $-1,1$. All four of these lead to $a(d_k-a+1)=2$, so we get $c_k(V)\geq 4$. So we either have $V=(L_2)_k$ or $V$ is a box product of copies of $L_1$. If there are $C$ copies of $L_1$, then $c_k(V)$ is the number of weights of the box product of the other $C-1$ copies of $L_1$ equal to either $0$ or $1$, which is $C-1\choose{\lfloor(C-1)/2\rfloor}$. This is at most 2 exactly when $C\leq 3$.
\end{proof}

\begin{lemma}\label{lem:deg1}
    If $V$ is Hesselink-type and $c_k(V)=1$ for some $k$, then exactly one irreducible summand of $V$ acts at $k$, and it is either $(L_1)_k$ or $(L_1)_k\boxtimes (L_1)_{k'}$ for some $k'\neq k$.
\end{lemma}

\begin{proof}
    Since $c_k$ behaves additively with respect to direct sums, by Lemma \ref{lem:cbound}, we have $c_k(V)=1$ if and only if exactly one irreducible summand of $V$ acts at $k$ and it is of the form $(L_1)_k\boxtimes(L_d)_{k'}$ for some $d$. If $d\geq 2$, then for all weights $a\geq 2$ of $L_d$, the weight $(-1,a)$ of $(L_1)_k\boxtimes(L_d)_{k'}$ contributes a term $(\chi_(L_1))_k(\chi_{L_{a-2}})_{k'}$ to $X_V$. Since $(L_1)_k\boxtimes(L_d)_{k'}$ is the only summand of $V$ that acts at $k$, these are the only terms with action at $k$ in $X_V$, so they do not cancel and $V$ is not Hesselink-type. So we are left with the cases $d=0$ and $d=1$, which are exactly the two cases in the lemma.
\end{proof}

Note that the cases in Lemma \ref{lem:deg1} are exactly isolated frames and leaves in the context of extended quiver representations. We are now ready to prove Theorem \ref{thm:sl2} \ref{thm:sl22}.

\begin{proof}[Proof of Theorem \ref{thm:sl2} \ref{thm:sl22}.]
    By the proof of Lemma \ref{lem:trivtype} \ref{lem:trivtype2}., removing a leaf from $V$ does not change $X_V$. Thus we may remove all leaves from $V$. Then by Lemma \ref{lem:deg1}, any vertex $k$ with $c_k(V)=1$ is an isolated frame $(L_1)_k$, which does not affect the rest of the vertices. These are extended quiver representations, so we may delete these as well.

    Now, there are no vertices $k$ with $c_k(V)=1$. So for $V$ to be Hesselink-type, we must have $c_k(V)\in\{0,2\}$ for all $k$. If $c_k(V)=0$, then by Lemma \ref{lem:cbound} no summands of $V$ act at $k$, so we can remove $k$ from $\g$. So we may assume $c_k(V)=2$ for all $k$.

    By Lemma \ref{lem:cbound} \ref{lem:cbound3}, all irreducible components of $V$ are now either frames, self-loops, edges, or 3-hyperedges. Both the self-loop and 3-hyperedge achieve $c_k(V)=2$ for all vertices $k$ involved, so they must appear as isolated components. Now, consider any isolated component consisting of frames and edges. Let the number of vertices in the component be $l'$, and the number of frames and edges be $a$ and $b$, respectively. Then we have $\sum_kc_k(V)=2l'=a+2b$. Since the component is assumed to be connected and there are no leaves, the only ways for this to be satisfied are if the component is a path with frames on each terminal vertex or a cycle.
\end{proof}

As a corollary of Theorem \ref{thm:sl2} and Lemma \ref{lem:cofree}, we obtain

\begin{corollary}
    If $\g$ is a product of copies of $\sl_2$, then every Hesselink-type $\g$-representation is cofree.
\end{corollary}

\section{More on Hesselink-type representations and q-character formulas}\label{sec:small}

In this section we compile a list of observations related to Hesselink-type representations and $\q$-character formulas beyond the scope of Theorem \ref{thm}. These include a geometric interpretation of Hesselink-type representations, $\q$-character formulas arising from Proposition \ref{prop:hesselink} for non-Hesselink-type representations, and some counterexamples to natural questions regarding Hesselink-type representations.

\subsection{Geometric interpretation}

Let us interpret $X_V$ geometrically. Note that by the Weyl character formula, $X_V$ is equal to a \emph{virtual} character of $\g$ (with coefficients in $\Z[\q]$). Given any $\lambda\in\Lambda$, replacing the $\prod_{\alpha\in S_+}(1-q_\alpha t^{-w\alpha})$ in $X_V$ with $t^{w\lambda}$ yields
\[\frac{\sum_{w\in W}(-1)^{\ell(w)}t^{w\rho}t^{w\lambda}}{\sum_{w\in W}(-1)^{\ell(w)}t^{w\rho}}.\]
When $\lambda$ is dominant, by the Weyl character formula, this is exactly the character $\chi_\lambda$. To determine what happens when $\lambda$ is not dominant, consider the shifted $W$-action on $\lambda$ given by $w\cdot\lambda=w(\lambda+\rho)-\rho$. It is easy to check that for all $y\in W$, we have
\[\frac{\sum_{w\in W}(-1)^{\ell(w)}t^{w\rho}t^{w\lambda}}{\sum_{w\in W}(-1)^{\ell(w)}t^{w\rho}}=\frac{\sum_{w\in W}(-1)^{\ell(w)}t^{w\rho}(-1)^{\ell(y)}t^{w(y\cdot\lambda)}}{\sum_{w\in W}(-1)^{\ell(w)}t^{w\rho}}.\]
In particular, if $y\cdot\lambda$ is dominant then this is $(-1)^{\ell(y)}\chi_{y\cdot\lambda}$. When there is no such $y$, $\rho+\lambda$ is not regular (i.e.~lies on a hyperplane $\langle\rho+\lambda,\alpha^\vee\rangle=0$ for some $\alpha\in R_+$), and thus we have
\[\frac{\sum_{w\in W}(-1)^{\ell(w)}t^{w\rho}t^{w\lambda}}{\sum_{w\in W}(-1)^{\ell(w)}t^{w\rho}}=0\]
since the numerator is $W$-antiinvariant and consists of nonregular weights. 

Using this, we can reinterpret the original expression using Borel-Weil-Bott. Namely, let $B\subset G$ be the Borel corresponding to our choice of Cartan subalgebra and polarization, and let $T\subset B$ be its maximal torus. Let $V_+\subset V$ be the direct sum of all weight spaces with weights in $S_+$, which has a structure as a $T$-module and hence as a $B$-module via the quotient map $B\to T$. Up to a change in variables $\q\mapsto -\q$, the $B$-module $\wedge^\bullet V_+^*\cong(\wedge^\bullet V_+)^*$ has $\q$-character given exactly by $\prod_{\alpha\in S_+}(1-q_\alpha t^{-\alpha})$, where the degree in $\q$ corresponds to the natural multigrading on the exterior algebra extending the multigrading on $V$. Thus by Borel-Weil-Bott, $X_V$ is (up to a change of variables $\q\mapsto-\q$) the $\q$-character of
\[R\Gamma(G\times^B(\wedge^\bullet V_+)),\]
where we view $G\times^B(\wedge^\bullet V_+)$ as a vector bundle on the flag variety $G/B$ and the $\q$-grading is via the grading on the exterior algebra as before. Here, when we take the $\q$-character of a derived $G$-module $X_\bullet$ we mean the Euler characteristic
\[\sum_{i\in\Z}(-1)^i\ch_\q H^i(X_\bullet).\]
We obtain
\begin{lemma}
    $V$ is Hesselink-type if and only if $[R\Gamma(G\times^B(\wedge^\bullet V_+^*))]\in\Z[\q][\C]$ in the Grothendieck group $K_0(D\mathrm{Rep}(G))[\q]$.
\end{lemma}

Note that the change of variables $\q\mapsto-\q$ does not affect this statement.

\subsection{Other examples}\label{sec:nonexamples}

Let us discuss formulas akin to Formula \eqref{eq:formula} for non-Hesselink-type representations that we can obtain using the machinery of Proposition \ref{prop:hesselink}. We omit the calculations that the examples below are not Hesselink-type; they are all elementary calculations that can be easily checked by hand.

In all cases below we will let $P_\q:\Lambda\to\Z[\q]$ be a function such that
\[M_\q(\lambda)=\sum_{w\in W}(-1)^{\ell(w)}P_\q(w(\lambda+\rho)-\rho),\]
and we have a formula for $\C[\N_V]$ analogous to Theorem \ref{thm}, but with $\P_\q,\M_\q$ replaced with $P_\q,M_\q$ as below. When $V$ is irreducible, the multigrading variable $\q$ is just a single variable, so we will denote it by $q$ instead and write $P_q,M_q$, etc.

Many of the examples will follow a trend where $P_\q$ is obtained by taking the definition of $\P_\q$ and replacing the sets $S_+,S_0,S_-$ with some other submultisets of $S$. That is, we use the replacements for the sets $S_+,S_0,S_-$ instead of the old ones in the definition of a weakly negative function, and then use that version of a weakly negative function in the definition of $\P_\q$.

We would like to point out that a number of formulas of this type were computed in \cite{wallachwillenbring}, where the authors found some cases where one can take $P_\q$ to be $\P_\q$, but with the sets $S_0,S_-,S_+$ replaced by $\varnothing$, either the set of negative roots of $\g$ or the set of negative short roots of $\g$, and $S\setminus(S_0\cup S_-)$, respectively. We refer the reader to their paper for details.

\begin{enumerate}
    \item $\g=\sl_2, V=L_3$. Since $V$ is simple, we have a single grading variable, so let us write $\q=q$. Then $P=(1-qt^{-1})(1-q^2-q^2t^{-2}+q^3t)$. $P_q$ consists of terms corresponding to functions $m:S\to\Z_{\geq 0}$ with $m(1)=0$, $\min(m(3),m(-1))=0, \min(m(3),m(-3))=0$. In other words, we have the union of linear combinations of $-1,-3$ and multiples of $3$.
    \item $\g=\sl_2, V=L_4$. Then $P=(1-q)(1-qt^{-2})(1-q^2-q^2t^{-2}+q^3t^2)$. $P_q$ consists of terms corresponding to functions $m:S\to\Z_{\geq 0}$ with $m(0)=m(2)=0$, $\min(m(4),m(-2))=0, \min(m(4),m(-4))=0$. In other words, we have the union of linear combinations of $-2,-4$ and multiples of $4$.
    \item $\g=(\sl_2)^2,V=L_2\boxtimes L_1$. Then $P=(1-qt^{(-2,-1)})(1-qt^{(0,-1)})(1-q^2)$. $P_q$ consists of terms corresponding to functions $m:S\to\Z_{\geq 0}$ with $m((2,1))=m((0,1))=0$, and $\min(m((2,-1)),m((-2,1)))=0$. In other words, this is the same formula as the old one except $S_+=\{(2,1),(0,1)\},S_0=\{(2,-1),(-2,1)\},S_-=\{(-2,-1),(0,-1)\}$.
    \item Let $\g$ be of type $B_2$, and let $\alpha_1,\alpha_2$ be the short and long simple roots, respectively. Choose a coordinate basis of $\h^*$ so that $\alpha_1=(2,0), \alpha_2=(-2,2)$. Let $\omega_1,\omega_2$ be the fundamental weights corresponding to $\alpha_1^\vee, \alpha_2^\vee$, so we have $\omega_1=(1,1), \omega_2=(0,2)$ using the above basis.
    \begin{enumerate}
        \item Let $V=L_{\omega_1}$. Then $\dim V=4$, and the weight spaces of $V$ have weights $(\pm 1,\pm 1)$. There are (at least) two valid choices of $P$, given by $P=(1-qt^{(1,-1)})(1-qt^{(-1,1)})(1-qt^{(-1,-1)})$ and $P=(1-qt^{(1,-1)})(1-qt^{(-1,-1)})$. These yield formulas for $\P_q$ with $S_-,S_0$ replaced by either $S_-=\{(-1,-1)\}, S_0=\varnothing$ or $S_-=\{(-1,-1),(1,-1)\},S_0=\varnothing$, respectively.
        \item Let $V=L_{\omega_2}$. Then $\dim V=5$, and the weights in $V$ are $(0,\pm 2), (\pm 2,0),(0,0)$. There are (at least) two valid choices of $P$ given by $P=(1-q)(1-qt^{(-2,0)})(1-qt^{(0,-2)})(1-qt^{(2,0)})$ and $P=(1-q)(1-qt^{(-2,0)})(1-qt^{(0,-2)})$. These yield formulas for $\P_q$ with $S_-,S_0$ replaced by either $S_-=\{(0,-2)\},S_0=\varnothing$ or $S_-=\{(0,-2),(-2,0)\},S_0=\varnothing$, respectively.

        Note that the formula with $S_-=\{(0,-2),(-2,0)\}$ recovers the formula in \cite{wallachwillenbring} for $G=SO(2n+1)$ and $V$ the tautological representation of $G$ in the case $n=2$. 
    \end{enumerate}
\end{enumerate}

\subsection{Counterexamples}\label{sec:counter}

As Hesselink-type representations are a new concept and rather mysterious, it is interesting to consider how they might relate to other classes of representations. We provide some counterexamples to some potential relations one may ask about. Some of the counterexamples are strictly stronger than others in terms of what statements they disprove, but we nonetheless note the weaker ones in hopes that they may be useful to someone else interested in studying Hesselink-type representations or $q$-character formulas for the nilpotent cone.

\begin{example}\label{eg:noncofree}
    Let $\g=\sl_3$ and let $U$ be the tautological representation of $\g$. Then let $V=U^{\oplus 3}\oplus U^*$ with $q_1,q_2,q_3$ corresponding to the $U$ summands and $q^4$ corresponding to the $U^*$ summand. Direct computation yields $X_V=1-q_1q_2q_3$, so $V$ is Hesselink-type. However, Schwarz and Wehlau proved in \cite[Theorem 2.8]{Schwarz1998InvariantsOF} that $V$ is coregular but not cofree. So not all Hesselink-type representations are cofree. Perhaps surprisingly, a direct Gr\"obner basis computation of the Koszul resolution of $\C[\N_V]=\C[V]/\C[V]^G_+\C[V]$ shows in this case that the formula in Theorem \ref{thm} still holds for $V$. This computation was carried out in Python.
\end{example}

\begin{example}\label{eg:noncoregular}
    Let $\g=\sl_3\times\sl_3$, and let $U_1,U_2$ be the tautological representations of the first and second copies of $\sl_3$, respectively. Then let $V=U_1^{\oplus 2}\oplus U_2^{\oplus 2}\oplus U_1\boxtimes U_2$. Let $q_1,q_2,q_3,q_4,q_5$ correspond to the summands $U_1,U_1,U_2,U_2,U_1\boxtimes U_2$, respectively. Direct computation yields $X_V=1-\prod_{i=1}^5q_i$, and a straightforward commutative algebra computation yields
    \[\dim_q\C[V^{\rho^\vee}]^\h=\frac{1-q_1q_2q_3q_4q_5^4}{(1-q_5^3)(1-q_1q_3q_5^2)(1-q_1q_4q_5^2)(1-q_2q_3q_5^2)(1-q_2q_4q_5^2)}.\]
    Thus Lemma \ref{lem:hesselinkid} implies we have
    \[\dim_q\C[V]^G=\frac{1-q_1q_2q_3q_4q_5^4}{(1-q_5^3)(1-q_1q_3q_5^2)(1-q_1q_4q_5^2)(1-q_2q_3q_5^2)(1-q_2q_4q_5^2)(1-\prod_{i=1}^5q_i)},\]
    so $V$ is not coregular. So not all Hesselink-type representations are coregular. Similar to the previous example, a direct Gr\"obner basis computation in Python shows that the formula in Theorem \ref{thm} still holds for $V$. 
\end{example}

\begin{example}
    Let $\g=\sl_2$, and $V=L_3$ be the irreducible representation with highest weight 3. Then $V$ is cofree\footnote{see \cite{littelmann}, this is also indirectly proved in \cite{vinberg} since the unique outer grading on $\sl_3$ yields $\g_0\cong\sl_2,\g_1\cong L_3$.} and we have $\rho=\rho^\vee=1,W=S_2,S_+=\{1,3\}$. The numerator of $X_V$ becomes $-q^2t^3+qt^2+t-t^{-1}-qt^{-2}+q^2t^{-3}$, whereas if $V$ was Hesselink-type then $X_V$ would have to lie in $(t-t^{-1})\Z[q]$. So not all cofree representations are Hesselink-type.
\end{example}

\begin{example}\label{eg:g2}
    Let $\tilde{\g}$ be a Lie algebra of type $G_2$ with simple roots $\alpha_1,\alpha_2$ where $\alpha_1$ is the long root and $\alpha_2$ is the short root. For $i=0,1$, let $\g_i$ be the direct sum of all weight spaces of $\tilde{\g}$ with weight $\lambda$ such that $\langle\lambda,\alpha_1^\vee\rangle\equiv i\pmod 2$. Then $\g_0,\g_1$ define an inner grading on $\tilde{\g}$. Let $\g=\g_0,V=\g_1$. Then we have $\g\cong\sl_2\times\sl_2, V\cong L_1\boxtimes L_3$. We have $S_+=\{(1,3),(-1,3),(1,1)\}$, so in particular we have $X_V=1+q\chi_{L_1\boxtimes L_1}-q^2\chi_{L_0\boxtimes L_4}$, so $V$ is not Hesselink-type. So a representation arising from a graded Lie algebra with inner grading is not necessarily Hesselink-type.
\end{example}

\begin{example}\label{eg:3232}
    As mentioned in Remark \ref{rem:>2}, consider the representation $V$ of the cyclic quiver with 4 vertices that has dimension vector $3,2,3,2$. As pointed out in the remark, the numerator of $X_V$ contains terms whose weights do not lie on a reflection hyperplane and are also not in $W\rho$. This is not possible if $X_V$ is a direct sum of trivial characters, so $V$ is not Hesselink-type. So not all representations of cyclic quivers are Hesselink-type.
\end{example}

\subsubsection{An extended counterexample}\label{sec:nonformula}

Let $\g=\sl_2\times\sl_3$, let $U$ be the tautological representation of $\sl_3$, and let $V=(L_2\oplus\C)\boxtimes U^*$. This example turns out to be Hesselink-type and cofree, but does not satisfy the $\q$-character formula in Theorem \ref{thm}.

Let $\q=(q_1,q_2)$ where $q_1$ corresponds to the submodule $L_2\boxtimes U^*\subset V$ and $q_2$ correponds to $\C\boxtimes U^*\subset V$. Let the fundamental weights of $\sl_3$ be $\omega_1,\omega_2$ so that $U$ has highest weight $\omega_1$. Direct calculation yields $X_V=1+q_1^2q_2$, so $V$ is Hesselink-type.

Let us now compute the $\g$-invariants in $\C[V]$. Let us view $V$ as $V\cong\Hom(U,L_2\oplus\C)$. By the first fundamental theorem of invariant theory (see \cite[Theorem 9.1.4]{procesi2006lie}), the $\sl_3$-invariants in $\C[V]$ are isomorphic to $\C[\wedge^3(L_2\oplus\C)]\cong\C[L_2\oplus\C]$ where the generators of the latter are in degree $3$. We then have $\C[L_2\oplus\C]^{\sl_2}\cong\C[L_2]^{\sl_2}\otimes\C[\C]$ is a polynomial algebra generated by generators $e_1,e_2$ in degrees $3$ and $6$. $e_1$ is the determinant of the $3\times 3$ minor in $V$ with rows corresponding to $L_2$, so it is irreducible. One can check that $e_2$ is not in the ideal generated by $e_1$, and thus $e_1,e_2$ form a regular sequence. This implies that $V$ is cofree.

On the other hand, direct computation yields that $V$ does not satisfy Formula \eqref{eq:formula}, with the first disparity occurring in degree $(4,1)$. Rather than writing down the explicit calculation, we will give a conceptual reason for this phenomenon.

Recall from Proposition \ref{prop:hesselink2} that if $V$ is Hesselink-type, cofree, and $\C[V^{\rho^\vee}]$ is free over $\C[V^{\rho^\vee}]^\h$, then the formula in Theorem \ref{thm} holds for $V$. In this case, we have $S_0=\{(2,-\omega_1),(-2,\omega_2),(0,\omega_1-\omega_2),(0,\omega_1-\omega_2)\}$. The repeated instance of $(0,\omega_1-\omega_2)$ prevents $\C[V^{\rho^\vee}]$ from being free over $\C[V^{\rho^\vee}]^\h$. Indeed, if the elements in $S_0$ are $v_1,v_2,v_3,v_4$, respectively, then $\C[V^{\rho^\vee}]^\h$ is generated by $p_1:=v_1v_2v_3$ and $p_2:=v_1v_2v_4$. Then one has the relation $p_1v_4=p_2v_3$, and tracing this relation through the calculations in the proof of Proposition \ref{prop:hesselink}, one sees that the relation does not vanish, and there is a disparity in the formula in degree $(4,1)$. Note that this is unlike the other case we have seen where $\C[V^{\rho^\vee}]$ is not free over $\C[V^{\rho^\vee}]^\h$ in the $l$-cycle for $\prod\sl_2$; in that case the syzygies of $\C[V^{\rho^\vee}]$ do vanish when computing their contribution to the $\q$-character formula.

\printbibliography

\footnotesize{
{\bf V.K.}: Department of Mathematics Harvard University and CMSA, 1 Oxford St, Cambridge, MA 02138, USA\\
\hphantom{x}\quad\, {\tt vkrylov@math.harvard.edu}, {\tt krylovasya@gmail.com}

{\bf F.W.}: Department of Mathematics, MIT,
77 Massachusetts Ave, Cambridge, MA 02139, USA\\
\hphantom{x}\quad\, {\tt fw0@mit.edu}}

\end{document}